\documentclass{amsart}

\usepackage{mathabx}
\usepackage{extpfeil}
\usepackage{hyperref}
\usepackage{amssymb, fancyhdr, stmaryrd, amsthm, mathrsfs, textcomp}
\usepackage[utf8]{inputenc}
\usepackage{amsfonts,amsmath,amscd,amssymb}
\usepackage{dsfont,mathtools}
\usepackage{xspace}
\usepackage{undertilde}
\usepackage{tikz-cd}
\usepackage{caption}
\usepackage{subcaption}
\usepackage{graphicx}
\usepackage{listings}
\usepackage{faktor}
\usepackage{calligra}
\usepackage{pifont}
\newcommand{\ccr}{\ding{57}}

\def\R{\ensuremath\mathbb{R}}

\def\bG{\ensuremath\mathbb{G}}
\def\bS{\ensuremath\mathbb{S}}
\def\Z{\ensuremath\mathbb{Z}}
\def\F{\ensuremath\mathbb{F}}
\def\K{\ensuremath\mathbb{K}}

\def\D{\ensuremath\mathcal{D}}
\def\O{\ensuremath\mathcal{O}}

\def\X{\ensuremath\mathcal{X}}

\def\a{\ensuremath\alpha}
\def\b{\ensuremath\beta}
\def\g{\ensuremath\gamma}
\def\s{\ensuremath\sigma}

\def\Om{\ensuremath\Omega}

\def\P{\ensuremath\mathcal{P}}
\def\H{\ensuremath\mathbb{H}}

\def\L{\ensuremath\Lambda}
\def\M{\ensuremath\mathcal{M}}

\def\x{\ensuremath\times}
\def\tensor{\ensuremath\otimes}
\def\to{\ensuremath\rightarrow}
\def\sset{\ensuremath\subseteq}

\def\<{\ensuremath\langle}
\def\>{\ensuremath\rangle}
\def\D{\ensuremath\Delta}
\def\quotient#1#2{%
    \raise1ex\hbox{$#1$} \big /\lower1ex\hbox{$#2$}%
}
\DeclareMathAlphabet{\mathcalligra}{T1}{calligra}{m}{n}
\DeclareMathOperator{\rank}{rank}

\DeclareMathOperator{\im}{im}

\DeclareMathOperator{\Hom}{Hom}
\DeclareMathOperator{\aut}{Aut}

\DeclareMathOperator{\Sy}{Sym}
\DeclareMathOperator{\ind}{Ind}

\DeclareMathOperator{\pd}{proj.dim}

\newtheorem{thm}{Theorem}[section]
\newtheorem{lemma}[thm]{Lemma}
\newtheorem{prop}[thm]{Proposition}
\newtheorem{cor}[thm]{Corollary}

\theoremstyle{definition}
\newtheorem{defn}[thm]{Definition}
\newtheorem{conj}[thm]{Conjecture}

\newtheorem{ex}[thm]{Example}

\newcommand{\va}{\ensuremath{\mathbf{a}}\xspace}

\newcommand{\vg}{\ensuremath{\mathbf{g}}\xspace}
\newcommand{\vh}{\ensuremath{\mathbf{h}}\xspace}

\newcommand{\vk}{\ensuremath{\mathbf{k}}\xspace}

\newcommand{\vp}{\ensuremath{\mathbf{p}}\xspace}
\newcommand{\vs}{\ensuremath{\mathbf{s}}\xspace}
\newcommand{\vt}{\ensuremath{\mathbf{t}}\xspace}
\newcommand{\vu}{\ensuremath{\mathbf{u}}\xspace}

\newcommand{\vw}{\ensuremath{\mathbf{w}}\xspace}
\newcommand{\vx}{\ensuremath{\mathbf{x}}\xspace}
\newcommand{\vy}{\ensuremath{\mathbf{y}}\xspace}
\newcommand{\vz}{\ensuremath{\mathbf{z}}\xspace}
\newcommand{\sA}{\mathcal{A}}
\newcommand{\sB}{\mathcal{B}}
\newcommand{\sC}{\mathcal{C}}
\newcommand{\sD}{\mathcal{D}}

\newcommand{\sE}{\mathcal{E}}
\newcommand{\sF}{\mathcal{F}}
\newcommand{\sG}{\mathcal{G}}

\newcommand{\sH}{\mathcal{H}}

\newcommand{\sM}{\mathcal{M}}
\newcommand{\sL}{\mathcal{L}}
\newcommand{\sN}{\mathcal{N}}
\newcommand{\sQ}{\mathcal{Q}}
\newcommand{\sP}{\mathcal{P}}

\renewcommand{\S}{\mathcal{S}ph(S)}
\newcommand{\sT}{\mathcal{T}}

\newcommand{\sX}{\mathcal{X}}

\newcommand{\mA}{\mathfrak{A}}
\newcommand{\mD}{\mathfrak{D}}

\newcommand{\mH}{\mathfrak{H}\mathfrak{u}\mathfrak{l}\mathfrak{l}}

\newcommand{\vsub}{\rotatebox[origin=c]{90}{$\subseteq$}}
\newcommand{\Irr}{{\mathrm I}{\mathrm r}{\mathrm r}}
\newcommand{\IrA}{\Irr (\F A)}
\newcommand{\Ch}{{\mathrm C}{\mathrm s}{\mathrm h}_G(\sX_\bullet)}
\newcommand{\ChA}{{\mathrm C}{\mathrm s}{\mathrm h}_{G,A}(\sX_\bullet)}
\newcommand{\ChD}{{\mathrm C}{\mathrm s}{\mathrm h}_{G,A}(\sD_\bullet)}
\newcommand{\ChAc}{{\mathrm C}{\mathrm s}{\mathrm h}_{G,A,\chi}(\sX_\bullet)}
\newcommand{\dCh}{{{\mathrm C}{\mathrm s}{\mathrm h}_G^{\circ}(\sX_\bullet)}}
\newcommand{\dChA}{{{\mathrm C}{\mathrm s}{\mathrm h}_{G,A}^{\circ}(\sX_\bullet)}}
\newcommand{\dChAc}{{{\mathrm C}{\mathrm s}{\mathrm h}_{G,A,\chi}^{\circ}(\sX_\bullet)}}
\newcommand{\Dn}{{\mathrm D}{\mathrm i}{\mathrm a}{\mathrm g}_n}
\newcommand{\Sh}{{\mathrm S}{\mathrm h}_G(\sX_\bullet)}
\newcommand{\Cn}{{\mathrm C}_n}
\newcommand{\Sn}{{\mathrm S}_n}
\newcommand{\GL}{{\mathrm G}{\mathrm L}_n(\K)}
\newcommand{\PGL}{{\mathrm P}{\mathrm G}{\mathrm L}_n(\K)}
\newcommand{\dSh}{{{\mathrm S}{\mathrm h}_G^{\circ}(\sX_\bullet)}}

\def\Vl{\ensuremath\underline{\underline{V}}}
\def\Vt{\ensuremath\utilde{\utilde{V}}}
\def\VtI{\ensuremath\utilde{\utilde{V^\sG}}}
\def\VlI{\ensuremath\underline{\underline{V^\sG}}}
\def\VtC{\ensuremath\utilde{\utilde{V_\sG}}}
\def\VlC{\ensuremath\underline{\underline{V_\sG}}}

\begin{document}
\title[Kac-Moody Groups \& Cosheaves]{Kac-Moody Groups and Cosheaves on Davis Building}
\author{Katerina Hristova}
\address{Department of Mathematics, University of Warwick, Coventry, CV4 7AL, UK}
\email{K.Hristova@warwick.ac.uk}

\author{Dmitriy Rumynin}
\address{Department of Mathematics, University of Warwick, Coventry, CV4 7AL, UK\newline
\hspace*{0.31cm}  Associated member of Laboratory of Algebraic Geometry, National
Research University Higher School of Economics, Russia}
\email{D.Rumynin@warwick.ac.uk}

\thanks{The research was partially supported
  by the Russian Academic Excellence Project `5--100' and by Leverhulme Foundation.
  The authors are indebted to Inna Capdeboscq for constant attention to
  our efforts. The authors would like to thank Peng Xu for valuable lessons
  on the work of Schneider and Stuhler
and Robert Kurinczuk for valuable correspondence.}
\date{April 12, 2018}

\begin{abstract}
We investigate
smooth representations of complete Kac-Moody groups. 
We approach representation theory via geometry, in particular, 
the group action on the Davis realisation of its Bruhat-Tits building.
Our results include an estimate on projective dimension, localisation
theorem,
unimodularity and homological duality.
\end{abstract}

\maketitle

Our investigation of representation theory of Kac-Moody groups
aims to combine two known lines of inquiry. Bernstein in his 1992
lectures in Harvard \cite{Bern} proposed to look at representation
theory of $p$-adic groups through a geometric prism, \`{a} la Klein.
A $p$-adic group $H$ acts on a space, its Bruhat-Tits building $\sB\sT$.
A careful study of this action brings new, useful insights into
representation theory of $H$. 
This approach culminated in the 1997 seminal work 
by Schneider and Stuhler \cite{SS2} where they developed 
a systematic approach for passing from representations 
to equivariant objects on $\sB\sT$, an ultrametric rendition 
of the Beilinson-Bernstein localisation. 

The second line comes from the 2002 influential work by 
Dymara and Januszkiewicz. They 
pioneered a method for computing cohomology of 
a Kac-Moody group $G$ by
studying cohomology of $\sB\sT$   
and its Davis realisation $\sD$ \cite{DyTa}.

In the present paper we examine the smooth representations of a
Kac-Moody group $G$ by localising them over   
$\sD$. In a certain sense, we unify the two lines of inquiry described above.
A natural question is whether it is possible to use $\sB\sT$ rather than
$\sD$.
It is possible only for those Kac-Moody groups $G$ that are hyperbolic
in the following sense: 
any proper Dynkin subdiagram is of finite type.
In particular, affine Dynkin diagrams are hyperbolic,
so our results are applicable to algebraic groups over local fields
and their Bruhat-Tits buildings. 

Let us now explain the content of the present paper.
We strive to cover the correct generality in our results, the
generality where our proofs work. The price we pay for this
is that the different sections of the paper have different
assumptions.
Let us go section by section explaining our results and our
assumptions.

In Section 1 we collect useful results about
Haar measure on a locally compact totally disconnected topological group $G$.
Most of this section is covered by Vigneras' book \cite[I.1-I.2]{Vig}
but we find it essential to set up the notation and
review some facts for the benefit of the reader.
A criterion for unimodularity (Proposition~\ref{unimod_cr}) is new.
Another accessible source is the book by Bushnell and Henniart
\cite{LL}, although they
assume unimodularity and work in characteristic zero.

In Section 2 we keep the same assumptions on the group $G$ as in
Section 1, in particular, $G$ is not necessarily unimodular.
In perspective, we would like to cover 
the group $\GL$ over a local
field $\K$. When $\GL$ acts on $\sB\sT$, 
the stabilisers are not  compact,  
just compact modulo centre. So we choose a central subgroup $A$ of  $G$,
modulo which we can effectively describe geometry and representation
theory.  In particular, we introduce 
the abelian category $\M_A (G)$ of 
$A$-semisimple smooth representations of $G$
over a field $\F$.
We
show that $\M_A (G)$
is equivalent to a category of representations
of a Hecke algebra (Proposition~\ref{equiv2}).
The pay-off is existence of enough projectives in $\M_A (G)$
(Corollary~\ref{enough_pro}). 

We study these projectives in Section 3 and 
contemplate projective resolutions. 
If $(P_\bullet, d_\bullet)$ is a resolution of
the trivial module, then  
$(P_\bullet \otimes V, d_\bullet \otimes  I_V)$
is a projective resolution of any module $V$ 
(Lemma~\ref{projres}).
At this point we prove our first main theorem \`{a} la Bernstein
(Theorem~\ref{mainthm}): 
if $G$ acts on a contractible simplicial set $\sX_\bullet$,
the projective dimension of $\M_A (G)$ is bounded above by 
the dimension of $\sX_\bullet$.
Interestingly enough, we could not find a discussion of group action
on a simplicial set in the literature, so we feel compelled
to include some deliberations on this topic.

In Section 4
we adopt the assumptions coming from Theorem~\ref{mainthm}:
the group (as before) $G$ acts on a simplicial (not necessarily
contractible)
set $\sX_\bullet$.
We investigate $G$-equivariant cosheaves and sheaves 
(also known as coefficient systems in homology
and cohomology) on $\sX_\bullet$. 
We prove our second main theorem \`{a} la Schneider-Stuhler
(Theorem~\ref{local}). 
It is a localisation theorem clarifying the interface between 
$\M_A (G)$ 
and $G$-equivariant cosheaves on $\sX_\bullet$. 

Since $\M (G)$ is a noetherian category, a finitely generated module
admits a finitely generated projective resolution.
However, the resolution $(P_\bullet, d_\bullet)$
in Theorem~\ref{mainthm} is not finitely generated.
The goal of Section 5 is to chase a construction of a finitely generated
resolution.
Such resolution for $p$-adic algebraic groups is constructed
by Schneider and Stuhler by choosing a suitable cosheaf
on the Bruhat-Tits building $\sB\sT_\bullet$.
A convenient abstract machinery for assembling such a cosheaf
comprises systems of idempotents,
is introduced by Meyer and Solleveld \cite{MeSo}. 
Inspired by these two approaches,
we  propose a similar construction in Conjecture~\ref{cat0thm},
proving only the 1-dimensional case in Theorem~\ref{tree}.
We lack several crucial tools available to Schneider and Stuhler.
Firstly, the Davis 
building $\sD_\bullet$ 
of a general type is not as well behaved as an affine $\sB\sT_\bullet$.
Secondly, we lack Bernstein's Theorem that certain subcategories
of $\M_A (G)$ are closed under subquotients \cite[Th. I.3]{SS2}.
To overcome these difficulties, we propose to utilise the metric
properties of  
$|\sD|$, which  is a CAT(0)-space by Davis' Theorem.
This controls the assumptions of Section 5: we work with a locally compact
totally disconnected $G$ acting on a simplicial set $\sX_\bullet$
whose geometric realisation $|\sX|$ admits a CAT(0)-metric.

Our assumptions naturally evolve in Section 6.
We assume that $G$ is {\em a topological group of Kac-Moody type},
i.e., it admits a generalised BN-pair with certain topological properties.
The main result of the section is Theorem~\ref{open},
a description of the Davis building $\sD_\bullet$ for such a group $G$.
Consequently, all results from previous sections are applicable to $G$.
Another important result is Theorem~\ref{unimod}:
a topological group of Kac-Moody type is unimodular. 

Notice that the Davis building is often called 
{\em the Davis realisation of Bruhat-Tits building} in the literature.
Our terminology is justified:
$\sB\sT_\bullet$ and $\sD_\bullet$ are distinct simplicial sets.
They are homotopic if the Dynkin diagram has no connected components
of finite type, but they are both homotopic to a point in this case.
Both of them can be obtained from the same chamber system,
yet by different means. 
While
$\sD$ is intimately connected with $\sB\sT$,
we feel they are quite distinct objects.

The Kac-Moody groups emerge in   the penultimate Section 7.
Given a root datum $\mD$,
we explain how the corresponding Kac-Moody group over a finite field
$G_\mD (\F_q)$
leads to a topological group of Kac-Moody type.
Further details and proofs are available
in a paper by Capdeboscq and Rumynin \cite{CaRu}. 

The final  Section 8 has  similar assumptions to Section 6.
We initiate the study of the homological duality for smooth $G$-modules.
Origins of homological duality go back to Hartshorne \cite{Har}.
For $p$-adic groups the duality was first introduced by Bernstein
and Zelevinsky \cite{BernZ}.
In our approach we are influenced by the work of Yekutieli 
on the duality for modules over noncommutative rings
\cite{Yek}
as well as Bernstein's lecture notes \cite{Bern}.
We formulate two conjectures \ref{conj_H} and \ref{conj_IM}
on homological duality at the end of this paper. 
We will address these conjectures in future research.

\section{Haar Measure for Totally Disconnected Groups} \label{zero}
Let $G$ be a locally compact totally disconnected topological group.
If $K$ is a compact open subgroup, we can choose a left Haar measure
$\mu_K$
on $G$
with $\mu_K(K)=1$. We denote the modular function by 
$\Delta:G\rightarrow \R^\times_{>0}$.

Now let $I$ be the set of indices $|K:C|$
of all compact open subgroups $C\leq K$.
Let $\Z_{(K)}$ be the ring of fractions on $\Z$ obtained by inverting
all numbers $n\in I$.
\begin{lemma} (cf. \cite[Lemma 2.4]{Vig}) 
\label{mu_value}
If $A\subseteq G$ is a Borel set, then 
$\mu_K (A)\in \Z_{(K)}\cup\{\infty\}$. Moreover, $\Delta (\vx)\in \Z_{(K)}$ for all $\vx\in G$.
\end{lemma}
\begin{proof}
The topology admits a basis at $e$ consisting of compact open subgroups
\cite[II.7.7]{HR}.
If $C$ is a compact open subgroup, then it is commensurable to $K$,
hence
$$
\mu_K (C) = \frac{|C:(C\cap K)|}{|K:(C\cap K)|} \in \Z_{(K)}.
$$
Since $A$ is a disjoint union of left cosets
of various compact open subgroups, 
$\mu_K (A)\in \Z_{(K)}\cup\{\infty\}$.
Finally, $\Delta (\vx) = \mu_K (K\vx)\in \Z_{(K)}$.
\end{proof}

Let $\F$ be a field of characteristic $p$ (possibly $p=0$)
equipped with the discrete topology.
We say that the field (or its characteristic) 
is {\em $K$-modular}, if $p$ divides the order $|K|$.
Similarly, it is {\em $K$-ordinary}, if $p$ 
does not divide $|K|$.
Recall that the order $|K|$ of a profinite group $K$
is a supernatural number
$\prod_p p^{n_p}$ with $n_p\in \{0,1,\ldots \infty \}$
that is the least common multiple of orders of $K/H$
for various open subgroups $H\leq K$.

A continuous function $\Theta: G\rightarrow \F$ is locally constant
and, consequently, smooth.
In fact, the sets of smooth functions, continuous functions and
locally constant functions coincide.

If  the characteristic $p$ is $K$-ordinary, then 
there is a natural ring homomorphism $\Z_{(K)}\rightarrow \F$.
Thus, by Lemma~\ref{mu_value}
we may think that the measure $\mu_K$
and the modular function $\Delta$ take values in $\F$.
In particular, 
given a compactly supported smooth function $\Theta:G\rightarrow \F$,
one can compute its integral 
$\int_G \Theta(\vx) \mu_K (d\vx) \in \F$. 

The $\F$-vector space 
$\sH=\sH(G,\F,\mu_K)$ 
of  
all compactly supported smooth functions 
is a commutative algebra under pointwise multiplication $\bullet$
and {\em the Hecke algebra} 
under the convolution product \cite{Vig,LL}:
$$
\Psi \star \Theta (\vx) = \int_G \Psi (\vy) \Theta (\vy^{-1}\vx) \mu_K (d\vy).
$$
This multiplication depends on the choice of
the compact open subgroup $K$ such that the field is $K$-ordinary.
If no such $K$ exists, there is no Hecke algebra as defined here.
If two such subgroups $K$ and $K^\prime$ are chosen,
the measures 
are scalar multiples
of each other:
$\mu_K = \alpha\mu_{K^\prime}$.
Hence, the corresponding Hecke algebras 
$(\sH,\star)$ and $(\sH,\star^\prime)$ are isomorphic:
$$
f(\Psi \star \Theta) = f(\Psi) \star^\prime f(\Theta)
\ 
\mbox{ where }
\
f(\Psi) = \alpha \Psi .
$$
The Hecke algebra $(\sH,\star)$ is associative but 
contains no identity unless $G$
is discrete.
The identity should be the delta-function at $e\in G$ but it is not
well-defined. 
Instead $\sH$ contains a family of idempotents approximating identity.
For a compact open subset $U$
we define a function
$\Lambda_U \in \sH$ by $\Lambda_U (\vx) = 0$ if
$\vx\not\in U$ and $\Lambda_U (\vx)= 1 / \mu_{K} (U)$
if
$\vx\in U$.
Now take a basis of topology at $e$ consisting of all compact open 
subgroups. Then the functions $\Lambda_K$
as $K$ runs over this basis of topology 
form 
a family of idempotents approximating identity.

It is convenient for computations when the group $G$ is unimodular.
If $G$ is not unimodular, 
the modular function shows up in 
the change of variables 
$\vy=\vx^{-1}$ 
\begin{equation}
\label{eq1}
\mu_K (d \vx) 
\xlongequal{\vy=\vx^{-1}}
\Delta (\vy) \mu_K (d \vy).
\end{equation}
Further properties of the modular function can be found
in the Vigneras' book \cite[I.2.7]{Vig}.
One of the following standard properties
\begin{itemize}
\item $G$ is compact modulo centre (in particular, compact),
\item $G$ is perfect (in particular, simple),
\item $G$ is second countable and admits a lattice,
\item $G$ admits a Gelfand pair (in particular, abelian) \cite[Prop 6.1.2]{vD}
\end{itemize}
ensures that the group $G$ is unimodular.
We finish with the following technical fact,
useful as a unimodularity criterion, which we will
use later in Theorem~\ref{unimod}:
\begin{prop}
\label{unimod_cr}
Consider a compact open subgroup $H$ of $G$ and $\vx\in G$.  Then
$$\Delta (\vx) \cdot |H: H \cap \vx^{-1}H\vx| = |H: H \cap \vx H\vx^{-1}|.$$
\end{prop}
\begin{proof} 
Since $\mu (H)=\mu_K (H)$ is finite,
it suffices to observe that
$
\Delta (\vx) \mu (H)
=
\mu (H\vx)
=$
$$
=
\mu (\vx^{-1} H\vx)
=
\frac{|\vx^{-1}H\vx: H \cap \vx^{-1} H\vx|}{|H: H \cap \vx^{-1} H\vx|}
\cdot \mu(H)
=
\frac{|H: H \cap \vx H\vx^{-1}|}{|H: H \cap \vx^{-1} H\vx|} \cdot \mu(H).
$$
\end{proof}

\section{Category of Smooth Representations} 
\label{one}

We study representations of a locally compact totally disconnected
topological group $G$ over a field $\F$.   
A representation $(\pi, V)$ of $G$ is called \textit{smooth} if 
for all $v \in V$ there exists a compact open subgroup $K_v$ of $G$ such that $\pi(\vk)v=v$ for all $\vk \in K_v$. We denote the abelian category of all smooth representations of $G$ by $\M(G)$.

Fix a closed central subgroup $A\leq G$, which could be trivial. 
We want to study $A$-semisimple smooth representations of $G$. A simple representation of $A$
 is just a simple $\F$-representation of the group algebra $\F A$. Hence, it is determined by a field extension $\widetilde{\F} \supseteq \F$ and a character $\chi : A \rightarrow \widetilde{\F}^\times$ such that $\widetilde{\F}$ is generated as an $\F$-algebra by the image of $\chi$. We denote this representation by $\widetilde{\F}_\chi$ and the set of such characters by $\IrA$.

\begin{defn} An \textit{$A$-semisimple smooth representation} of $G$ 
is a smooth representation $(\pi, V)$ which is semisimple as a
representation of $A$. By $\M_A (G)$ 
we denote the abelian category of 
$A$-semisimple smooth representations of $G$. For each character $\chi\in\IrA$
we denote by $\M_{A,\chi} (G)$   
the  full subcategory  $\M_{A} (G)$ 
of those representations that are direct sums of 
$\widetilde{\F}_\chi$ as representations of $A$.
\end{defn}

Now let $H$ be a closed subgroup of $G$ with $A\leq H$.
Then $H$ is also locally compact and totally disconnected. 
There are several ways of inducing a representation 
from $H$ to $G$. We quickly recall them. 

Let $(\s, W) \in \M_A(H)$. 
Consider the $\F$-vector space $\widehat{W}$ of all 
$H$-equivariant functions 
$f: G \to W$. 
Equivariance means that
\begin{itemize}
\item[(i)] $f(\vh \vg)=\s(\vh)f(\vg)$, for all $\vh \in H$ 
and $\vg \in G$.
\end{itemize}
Consider the $\F$-vector subspace $\widetilde{W}\subseteq\widehat{W}$
of all ``smooth'' functions, i.e.,  
\begin{itemize}
\item[(ii)] $f \in \widetilde{W}$
if and only if  
there exists a compact
open subgroup $K_f$ of $G$ such that $f(\vg \vk)=f(\vg)$, 
for all $\vg \in G$ and $\vk \in K_f$.
\end{itemize}
Consider the homomorphism 
$\rho: G \to\mbox{Aut}_\F(\widehat{W})$ given by $[\rho(\vg)f](\vg')
= f(\vg'  \vg)$ for $\vg,\vg' \in G$ and $f \in \widehat{W}$. 
If $f \in \widehat{W}$ and $\va\in A$, then
$[\rho(\va)f](\vg)
= f(\vg\va)= f(\va\vg)=\s(\va)f(\vg)$ for all $\vg\in G$.
Writing $W=\oplus_i W_i$ as a direct sum of simple $A$-modules
$W_i = \widetilde{\F}_{\chi_i}$,
we can present  $f=\sum_i f_i$ as a sum of $A$-equivariant
smooth functions $f_i: G \to W_i$ so that  
$[\rho(\va)f](\vg)
= \sum_i \s(\va)f_i(\vg)= \sum_i [\chi_i(\va)f_i](\vg)$.
This proves that $(\rho, \widehat{W})$ is $A$-semisimple
(but not smooth). 
Its submodule $(\rho, \widetilde{W})$ is smooth and also
$A$-semisimple, hence it is in $\M_A(G)$. 
Following standard conventions in the literature, we call the pair 
$(\rho, \widetilde{W})$ the representation of $G$ 
\emph{smoothly induced by $\s$} and denote it  Ind$_{H} ^{G} (\s)$. 

If we restrict our attention to the subspace 
of $\widetilde{W}$ of compactly supported modulo $H$ functions, 
we obtain another representation of $G$ called 
\emph{compactly induced} and denoted by $c-\ind_H^G(\s)$.

The $\F H$-module $\F G \underset{\F H}{\otimes} W$ becomes an $\F
G$-module by setting $\vg (\vg' \tensor w) =\vg\vg' \tensor w$ for
$\vg, \vg' \in G$, $w \in W$. 
We call this representation of $G$ \emph{algebraically induced} and 
denote it $a-\ind_H^G(\s)$.
It is $A$-semisimple but not, in general, smooth.

Let $H$ be open, $A\leq H$. This guarantees smoothness of $a-\ind_H^G(\s)$. 
Now consider the map 
$\varphi:\F G \underset{\F H}{\tensor}  W \to
\text{Fun}_H (G_H, {}_{H}W)$ 
given by 
$$
\vg \tensor w \mapsto (f: \vg \vh^{-1} \mapsto \vh w),
\quad \vg \in G, \vh \in H, w \in W.  
$$
As $H$ is open,
$\varphi$ is an isomorphism from $a-\ind_H^G(\s)$ to $c-\ind_H^G(\s)$.
%
%
Let us summarise the observations above:
\begin{lemma} \label{functors}
  Let $G$ be a locally compact totally disconnected group.
  Suppose $H\geq A$ is a subgroup of $G$, closed and compact modulo $A$. The following hold:
\begin{enumerate}
\item $\ind_H^G$ and $c-\ind_H^G$ define functors from $\M_A(H)\ (\text{or}\ \M_{A,\chi}(H))$ to $\M_A(G)$
$ (\M_{A,\chi}(G)\ \text{correspondingly})$. 
\item In the case when $H$ is also open, $a-\ind_H^G$ also defines a functor from $\M_A(H)\ (\text{or}\ \M_{A,\chi}(H))$ to $\M_A(G)\ (\M_{A,\chi}(G)\  \text{correspondingly})$. 
\end{enumerate}
\end{lemma}

\begin{lemma} \label{s-s}
Let $H$ be a subgroup of $G$, compact modulo
$A$. 
Suppose that the field $\F$ is
$H/A$-ordinary. 
Then the categories $\M_{A, \chi} (H)$ and $\M_A(H)$ are semisimple.
\end{lemma}
\begin{proof}
Let $V \in \M_A(H)$. Then by definition $V$ is $A$-semisimple and hence can be decomposed as $V=\underset{\chi}{\bigoplus} V_{\chi}$ with $V_\chi=\{ v \in V \ |\ \va v= \chi(\va)v \text{ for all }\va \in A \}$. In other words, $\M_A(H) = \underset{\chi}{\bigoplus} \M_{A, \chi}(H)$, so it is enough to prove the statement for $\M_{A, \chi}(H)$.

Let $V \in \M_{A, \chi}(H)$.
Then $V$ is an $\widetilde{\F}$-vector space with an $\widetilde{\F}$-linear $H$-action. Let $v \in V$. By smoothness there exists a compact open subgroup $K_v$ of $H$ such that $\vk v=v$ for all $\vk \in K_v$.
Let $V' \coloneqq \< Hv\>_{\widetilde{\F}}$.
Clearly, $H/AK_v$ is both compact and discrete.
Hence, $H/AK_v$ is finite and
$V'$ is a finite dimensional $\widetilde{\F}$-subspace of $V$.

We want to show that $V$ is $H$-semisimple.
It suffices to find a direct $\widetilde{\F}H$-complement in $V$
of a finite dimensional $H$-submodule $W$. 
Pick an $\widetilde{\F}$-linear projection $p: V \to W$. 
Since $W$ is finite dimensional, we can write 
$p(v)=\sum\limits_{i=1}^n p_i(v) e_i$ 
for some basis $e_1,\ldots ,e_n$ of $W$ and some
linear functions  $p_i: V \to \F$.

Pick a section $\vx\mapsto \dot{\vx}$ of the quotient
homomorphism $H\to H/A$.
Let $\mu=\mu_{H/A}$ be a Haar measure on $H/A$
with values in $\F$.
Define $\widehat{p}: V \to W$ by
$$ \widehat{p}(v) := \underset{H/A}{\int} \dot{\vx}^{-1} p( \dot{\vx} v) \mu(d\vx). $$

The map $\widehat{p}$ is well-defined:  
write $\dot{\vx}^{-1} p( \dot{\vx} v) = \sum\limits_{i} \sum
\limits_{j} \psi_{ij}(\vx^{-1})\varphi_j(\vx) e_i$ 
for some $\psi_{ij}, \varphi_i \in C^{\infty}(H, \widetilde{\F})$,
then integrate the functions.

Clearly, $\widehat{p}$ is a well-defined 
$\widetilde{\F}$-linear projection. Let us verify
that $\widehat{p}(\vy v) = \vy \widehat{p}(v)$ for all $\vy\in H$,
$v\in V$. Let $\overline{\vy}=\vy A\in H/A$. For the standard argument
we need a change of variable $\vz=\vx \overline{\vy}$. The group 
$H/A$ is compact, hence, unimodular and $\mu(d\vz)=\mu(d\vx)$.
Then $\dot{\vx}\vy = \va_\vx \dot{\vz}$ for some element $\va_\vx\in A$
depending on $\vx$ (we think that $\vy$ is fixed).
Furthermore, $\dot{\vx}^{-1} =  \va_\vx^{-1} \vy \dot{\vz}^{-1}$ and 
$$ \widehat{p}(\vy v) 
= \underset{H/A}{\int} \dot{\vx}^{-1} p( \dot{\vx} \vy v) \mu(d\vx)
= \underset{H/A}{\int} \va_\vx^{-1} 
\vy \dot{\vz}^{-1} p( \va_\vx \dot{\vz} v) \mu(d\vz)
=\vy \widehat{p}(v).
$$ 
The last equality holds because $\va_\vx$ acts via the scalar
$\gamma (\va_\vx) \in \widetilde{\F}$ and $p$ is 
$\widetilde{\F}$-linear.
This yields a decomposition $V=W \oplus \ker(\widehat{p})$, 
finishing the proof.
\end{proof}

If $A$ is trivial and hence $H$ is compact, then the category $\M(H)$ of smooth representations of $H$ is semisimple.

The Hecke algebra $\sH=\sH (G,\F, \mu_K)$, defined in the last section
is a $G-G$-bimodule, smooth on both left and right.
We turn these into two commuting with each other structures
of a left $G$-module:
$$
\,^\vx\psi (\vy) = \psi (\vx^{-1}\vy), \ \ 
\psi^\vx (\vy) = 
\psi (\vy\vx).
$$

Let $(M,*)$ be an $\sH$-module.  
$M$ is called \emph{smooth} if $\sH * M=M$. This is equivalent to saying that for every $m \in M$ there exists a compact open subgroup $K$ of $G$ such that $\Lambda_K * m=m$. All smooth $\sH$-modules form a category which we denote by $\M(\sH)$.

\begin{prop} \cite[I.4.4]{Vig}(cf. \cite[1.4.2]{LL}) 
  \label{equiv1}
  If $\sH$ exists (which follows from existence of a compact
  open subgroup $H$ such that the field $\F$ is $H$-ordinary),
  then the functor
  $$\sF: \M(G) \to \M(\sH), \
  \sF \Big( (\pi,V) \Big) \coloneqq (\varpi,V), \ 
\varpi(\Theta) v  
=  
\int_G \Theta(\vg) \pi(\vg) v \mu (d\vg)
$$
for all 
$\Theta \in \M(\sH)$, $v \in V$, 
is an equivalence of categories.
\end{prop}
Using this functor $\sF$, 
we define
$$
\M_{A, \chi}(\sH) \coloneqq \overline{\sF (\M_{A, \chi}(G))}, \ \ 
\M_{A}(\sH) \coloneqq \overline{\sF (\M_{A}(G))},
$$
i.e., the full subcategories of objects isomorphic to
the objects $\sF \Big( (\pi,V) \Big)$
with $(\pi,V)$ from the corresponding subcategories.
The following statement is a tautology, yet we articulate it because it is an
important
stepping stone.
\begin{prop} \label{equiv2}
  If $\sH$ exists,
  then $\M_{A} (G)$ is equivalent to $\M_{A}(\sH)$.
\end{prop}
Pick a module $V \in \M(G)$.
Its (skew) coinvariants $V_{A, \chi}$
is a module in 
$\M_{A, \chi}(G)$: 
$$
V_{A,\chi} \coloneqq \widetilde{\F}\otimes_{\F A} V,\ \
\mbox{where the ring homomorphism is}\ \ 
\chi: \F A \rightarrow \widetilde{\F}.
$$
Observe that if $V \in \M_{A, \chi}(G)$, then
$V$ is naturally a vector space over
$\widetilde{\F}$
and
$V$ and $V_{A,\chi}$ are naturally isomorphic.
Furthermore, the skew coinvariants define a functor
$$
\M(G)\longrightarrow \M_{A, \chi}(G), \
V \mapsto V_{A, \chi}, \ 
\varphi \mapsto \varphi_{A, \chi} = 1 \otimes \varphi ,  
$$
left adjoint to the inclusion functor  $\M_{A, \chi}(G)\longrightarrow
\M(G)$.
Applying equivalence $\sF$ and $\sG$ from Proposition~\ref{equiv1},
we get a corresponding skew invariants functor
$\M(\sH)\longrightarrow \M_{A, \chi}(\sH)$,
left adjoint to the inclusion functor  
$\M_{A, \chi}(\sH)\longrightarrow\M(\sH)$.
We can use these functors  
to show that $\M_{A,\chi}(\sH)$ has enough projectives. 

\begin{lemma} \label{enoughproj}
The category $\M_{A,\chi}(\sH)$ has enough projectives.
\end{lemma}
\begin{proof}
For $N \in \M_{A,\chi}(\sH)$, $n\in N$ 
we can define a map $\varphi:\sH \Lambda_H \to N$ 
by $\varphi (\Theta \Lambda_H) = \Theta * n$
once we choose a compact open subgroup $H$ such that
$\Lambda_H * n = n$.
The corresponding map
$\varphi_{A,\chi}:(\sH \Lambda_H)_{A,\chi} \to N$ 
has $n$ in its image.

It remains to observe that $(\sH \Lambda_H)_{A,\chi}$ is projective.
The module $\sH \Lambda_H$ is projective in $\M(\sH)$ \cite[I.5.2]{Ren}.
Hence, $(\sH \Lambda_H)_{A,\chi}$ is projective in $\M_{A,\chi}(\sH)$
because a functor (coinvariants in our case), left adjoint 
to a right exact functor (the embedding  in our case) takes projective objects
to projective objects.
%
\end{proof}

\begin{cor}
\label{enough_pro}
If $\sH$ exists, the categories $\M_{A, \chi}(G)$ and $\M_{A}(G)$ have enough projectives.
\end{cor}
\begin{proof}
The statement about $\M_{A, \chi}(G)$ is immediate.
The category 
$\M_A(G)$ is a direct sum $\oplus_\chi\M_{A, \chi}(G)$, hence, $\M_A(G)$ also has enough projectives.
\end{proof}

\section{Projective Dimension and Actions on Simplicial Sets}
\label{one+}

Let us now investigate the projective dimension of the category $\M_A(G)$.
As we have seen in the previous section,
induction and compact induction are useful functors.
\begin{lemma} (cf. \cite[I.5.9]{Vig})
Let $G$ be a locally compact totally disconnected group.  
Suppose $H\geq A$ is a subgroup of $G$, closed and compact modulo $A$.
Then $\ind_H^G$ takes injective objects to injective objects.
If $H$ is open
then $c-\ind_H^G$ takes projective objects to projective objects.

Moreover, if the field $\F$ is
$H/A$-ordinary and $(\s, W) \in \M_A(H)$, 
then $\ind_H^G (\s)$ is an injective object
and $c-\ind_H^G(\s)$ is a projective object, as soon as $H$ is open.
\end{lemma}

\begin{proof}
  Frobenius reciprocity for $\ind_H^G$ tells us that it is right adjoint to Res$_H^G$ and since any right adjoint to a left exact functor takes injective objects to injective objects.
  Similarly, by Frobenius reciprocity for compact induction from open $H$,
  $c-\ind_H^G$ is left adjoint to the restriction functor Res$_H^G$, which is exact. Any such functor takes projective objects to projective objects.

In the case when $\F$ is $H/A$-ordinary,
$\M_A(H)$ is  semisimple,
hence  $(\s, W) \in \M_A(H)$ 
is a semisimple $H$-module. In other words, $W$ is both injective and projective.
We are done by the first part. 
\end{proof}

Observe that $a-\ind_H^G \cong c-\ind_H^G$ for an open $H$. 
Therefore, we can deduce the following:

\begin{cor}
\label{projective}
Let $G$ be a locally compact totally disconnected group.
Suppose $H\geq A$ is a subgroup of $G$, open and compact modulo $A$.
Further suppose that the field $\F$ is $H/A$-ordinary. 
If $(\s, W)$ is a representation in $\M_A(H)$, then $\F G \underset{\F H}{\otimes} W$ is a projective object in $\M_A(G)$. The statement is also true if we replace $\M_A(H)$ with $\M_{A,\chi}(H)$.
\end{cor}

If $A$ is trivial, Corollary~\ref{projective} yields that
smooth representations of $G$ algebraically induced from a compact open subgroup are projective.

\begin{lemma}
\label{projres}
Let $G$ be a locally compact totally disconnected group. Suppose $\F$ is the trivial representation of $G$ and 
\begin{alignat*}{2}
        0 \rightarrow P_n\rightarrow P_{n-1} &\rightarrow \cdots &\mathllap{\cdots}\rightarrow P_0 \rightarrow \F \rightarrow 0
    \end{alignat*}
  is a projective resolution of $\F$ in $\M_A(G)$.
Let $(\pi, V)\in \M_A(G)$, not necessarily finite dimensional. Then 
 \begin{alignat*}{2}
        0 \rightarrow P_n \tensor V \rightarrow P_{n-1} \tensor V &\rightarrow \cdots &\mathllap{\cdots}\rightarrow P_0 \tensor V \rightarrow V\rightarrow 0
    \end{alignat*}
is a projective resolution for $V$ in $\M_A(G)$. The statement is also true if we replace $\M_A(G)$ with $\M_{A,\chi}(G)$.
\end{lemma}

\begin{proof}
We will prove the statement for $\M_A(G)$, but the proof is the same for $\M_{A, \chi}(G)$. 
For the result to hold, it is enough to show that $P_i \otimes V$ is a projective object in $\M_A(G)$ for all $i=1,..,n$.


Observe that
$\Hom_{\M_A(G)} (P_i \otimes V, \underline{\quad}\ ) \cong  \Hom_{\M_A(G)} (P_i, \Hom_{\F} (V, \underline{\quad}\ ))$:  
to every $\a\in\Hom_{\M_A(G)} (P_i \tensor V, W)$ we associate
$\b \in \Hom_{\M_A(G)} (P_i, \Hom_{\F} (V, W ))$
defined by $\b: p_i \mapsto (\g: v \mapsto \a(p_i \otimes v))$ for $p_i \in P_i, v \in V$. 
Conversely, to every $\b: p_i \mapsto (\g: v \mapsto w)$ we associate
$\a: p_i \otimes v \mapsto \b(p_i)(v)$ for $p_i \in P_i, v \in V, w \in W$.

Since $P_i$ is projective,
the functor $\Hom_{\M_A(G)} (P_i, \underline{\quad}\ )$ is exact.
As $V$ is a free $\F$-module $\Hom_{\F} (V, \underline{\quad}\ )$ is also exact. The composition of two exact functors is exact, so $\Hom_{\M_A(G)} (P_i \otimes V, \underline{\quad}\ )$ is exact
and $P_i \otimes V$ is projective.   
\end{proof}

Let $\sX_\bullet=(\sX_n)$, for $n=0,1,...$,
be a simplicial set \cite[Ch. 1]{GeMa}.
If $f:[m] \to [n]$ is nondecreasing map, $[n]=\{0,1,...n \}$,
by $\sX(f): \sX_n \to \sX_m$ we denote the $f$-th face map.
We say that $G$ \emph{acts} on the simplicial set $\sX_\bullet$
if $G$ acts continuously on each discrete set $\sX_n$
and the action {\em respects} the face maps $\sX(f)$
so that $G$ acts continuously on the geometric realisation $|\sX|$.
Using the canonical bijection \cite[I.2.9]{GeMa} 
\begin{equation}\label{can_bi}
\mathring{\tau}: \coprod_n
\mathring{\D}_n\times \sX_{(n)}
\rightarrow |\sX|
\end{equation}
where $\sX_{(n)}$ is the set of non-degenerate $n$-simplices
and $\mathring{\D}_n =\{(\alpha_0,\ldots \alpha_n)\in \R^{n+1}_{>0} | \sum \alpha_k =1\}$
is the abstract $n$-simplex, we can write this action by
\begin{equation} \label{act}
\vg \cdot \big( (\alpha_i), x \big)=
\big( F(\vg,x)(\alpha_i), \vg \cdot x \big) 
\end{equation}
where $F(\vg,x)$ is an auto-homeomorphism of the abstract $n$-simplex. 

The respect of the face maps does not necessarily mean
that the action commutes with the face maps $\sX(f)$.
Recall the standard notation \cite{GeMa}:
$\partial^i=\partial^i_n : [n-1] \rightarrow [n]$ is
the unique increasing map, missing the value $i$, 
$\sigma^i = \sigma^i_n : [n+1] \rightarrow [n]$ is
the unique non-decreasing surjective map,
assuming the value $i$ twice. 
Given $x\in\sX_n$, its codimension one faces
are $\sX(\partial^i)(x)$ for various $i$.
The codimension one faces of $\vg \cdot x$
are $\vg \cdot \sX(\partial^i)(x)$ but their order could be different.
Let us define the maps
\begin{equation} \label{act2}
R=R_n : G\times \sX_n \rightarrow S_{n+1} = \Sy ([n])
\ \ \mbox{ by } \ \ 
\vg \cdot \sX(\partial^i)(x)
=
\sX(\partial^{R(\vg,x)(i)})( \vg \cdot x).
\end{equation}

The algebraic condition on $R$ that allows the $G$-action on $|\sX|$
is that $(G\times\sX, R)$ constitutes {\em a crossed simplicial groupoid}.
Since we cannot find it written out, we give further details.
Firstly, $R$ must be {\em a groupoid map}: 
\begin{equation} \tag{\ccr 1}\label{cond_1}
R(1,x) = 1, \
R(\vg\vh, x) =
R(\vg, \vh \cdot x)
R(\vh, x)
\ \mbox{ for all } \
\vg,\vh\in G, \; x\in\sX_n.
\end{equation}
Now we need the symmetric crossed
simplicial group $\bS_\bullet$
\cite{FL}. 
Recall that it is a simplicial set with
$
\bS_n = S_{n+1}
$
and the face maps generated by
$$
\bS (\partial^i_n) ( \phi ) =
\sigma^i_{n-1}\circ \phi \circ \partial^{\phi^{-1}(i)}_n
, \ \  
\bS (\sigma^i_n) ( \phi ) (k)=
   \begin{cases}
     i, & \text{if}\ i=\phi(k), \\
     i+1, & \text{if}\ i=\phi(k-1), \\
     (\sigma^i_n)^{-1} \phi \, \sigma^{\phi^{-1}(i)}_n (k), & \text{otherwise}.
      \end{cases}
   $$
   Secondly, the map $R$ must be simplicial:
\begin{equation} \tag{\ccr 2}\label{cond_2}
\bS (f) (R_n(\vg, x)) =
R_m(\vg, \sX(f)(x))
\ 
\mbox{ for all }
\ 
f:[m] \to [n], \;
\vg \in G, \; x \in \sX_n.
\end{equation}
Notice that it suffices to verify Condition~(\ref{cond_2})
only for $f=\partial^i_n$ and $f=\sigma^i_n$ for all $i$ and $n$.  
Thirdly, the maps $R$ must compute the permutations of
codimension one faces as prescribed by Equation~(\ref{act2}):
\begin{equation} \tag{\ccr 3}\label{cond_3}
\vg \cdot \sX(\partial^i_n)(x)
=
\sX(\partial_n^{R(\vg,x)(i)})( \vg \cdot x)
\ \mbox{ for all } \
\vg\in G, \; x\in\sX_n, \; i \in [n]
\end{equation}
Finally, a similar condition must hold for the codimension one 
degenerations:
\begin{equation} \tag{\ccr 4}\label{cond_4}
\vg \cdot \sX(\sigma^i_n)(x)
=
\sX(\sigma_n^{R(\vg,x)(i)})( \vg \cdot x)
\ \mbox{ for all } \
\vg\in G, \; x\in\sX_n, \; i \in [n].  
\end{equation}
If the maps $R_n$ are independent the second argument $x\in\sX_n$,
then all these conditions are equivalent to saying that $G$,
turned to the trivial simplicial group $\bG_\bullet$ with
$\bG_n=G$, $\bG(f)=\mbox{Id}_G$, is a crossed simplicial group $\bG$ \cite{FL}.
We summarize this discussion in the following proposition.
Since we are not using it, we leave a proof out for an inquisitive reader.
\begin{prop} \label{cross} (cf. \cite[Prop 1.7]{FL}) 
Let $\sX_\bullet$ be a simplicial set with an abstract group $G$ acting on each $\sX_n$.
Given a system of functions $R=R_n : G\times \sX_n \rightarrow S_{n+1}$,
the following two statements are equivalent:
\begin{enumerate}
\item
  \begin{itemize}
  \item  The group $G$ acts on the topological space $|\sX|$ by
    the following simplification of Formula~(\ref{act}):
    $$
    \vg \cdot \big( (\alpha_i), x \big)=
\big( (\alpha_{R(\vg,x)(i)}), \vg \cdot x\big) 
    $$
  \item the $G$-action on degenerate simplices agree with Condition~(\ref{cond_4}),
  \item the maps $R$ are given by Formula~(\ref{act2}).
    \end{itemize}
\item The maps $R$ satisfy Conditions (\ref{cond_1}), (\ref{cond_2}), (\ref{cond_3})
  and (\ref{cond_4}).
  \end{enumerate}
\end{prop}

%

We are finally ready for the main theorem of this section,
whose idea goes back to Bernstein.

\begin{thm} (cf. \cite[IV.4.2]{Bern})
\label{mainthm}
Let $G$ be a locally compact totally disconnected group,
$A$ its closed central subgroup.
Suppose 
$G$ acts continuously on an $n$-dimensional simplicial set $\sX_\bullet$
with contractible geometric realisation $|\sX |$
so that $A$ acts trivially on $\sX_\bullet$. 
Suppose that the action of $G$ extends to $|\sX |$
(as in Proposition~\ref{cross}). 
Suppose further that the stabiliser $G_{x}$ of any
non-degenerated simplex $x\in \sX_{(k)}$ is 
not only open (that follows from continuity) but also compact modulo $A$. 
If the field  $\F$ is 
$G_{x}/A$-ordinary for any $x\in \sX_k$, then 
$$ \pd(\M_{A, \chi}(G)) \leq n \quad \text{and} \quad \pd(\M_{A}(G)) \leq n. $$
\end{thm}

\begin{proof}
Recall that the projective dimension of an object is the minimal length of a resolution by projective objects. Since $\M_{A, \chi}(G)$ and $\M_A(G)$ have enough projectives, projective resolutions exist, so we can talk about the projective dimension of the categories.


The simplicial homology complex of $\sX$ 
\begin{equation}\label{boundary_s3}
d_k :C^\sharp_k (\sX_\bullet , \F)
\rightarrow
C^\sharp_{k-1} (\sX_\bullet , \F), \ 
d_k \Big(\sum_{x\in \sX_k} \alpha_x x \Big) \coloneqq 
 \sum_{x\in \sX_k} \sum_{i=0}^k
(-1)^i
\alpha_x \, 
\sX(\partial^i_k)(x)
\end{equation}
is a complex of smooth $G$-modules
in  $\M_{A, 1}(G)$ under
\begin{equation}\label{action_s3}
\vg \cdot \Big(\sum_{x\in \sX_k} \alpha_x x \Big) \coloneqq 
 \sum_{x\in \sX_k} 
(-1)^{sign(R(\vg,x))}
\alpha_x \, (\vg \cdot x) \, .
\end{equation}
If $x= \sX(\sigma^i) (y)$ then
$y = \sX(\partial^i) (x) = \sX(\partial^{i+1}) (x)$
while all the other faces $\sX(\partial^j) (x)$
are degenerate.
Hence, $d(x)$ is a linear combination of degenerate simplexes
and the spans of degenerate simplexes form a subcomplex
of submodules 
$(C^\flat_k (\sX_{\bullet} , \F),d_k )$.

Let $X_k=C_k(\sX_{\bullet} , \F) \coloneqq C^\sharp_k(\sX_{\bullet} , \F) / C^\flat_k (\sX_{\bullet} , \F)$. 
The $\F$-vector space $X_k$ has a basis $[x]$ with various non-degenerate simplices
$x\in \X_{(k)}$. It is still a smooth $G$-module
in  $\M_{A, 1}(G)$. The spaces $X_k$ comprise the chain complex
 \begin{alignat*}{2}
   \mathscr{C}: \quad       X_{n} \xrightarrow{d_n} X_{n-1} &\xrightarrow{d_{n-1}} \ldots
   &\mathllap{\ldots}\xrightarrow{d_1} X_0 \ ,
 \end{alignat*}
 that computes the homology of $|\sX|$. 
 Since $|\X|$ is contractible, 
 all homology groups are trivial
 except $H_{0}(\mathscr{C})\cong\F$. This yields the exact sequence:
\begin{equation}\label{exa}
  0\rightarrow X_{n} \xrightarrow{d_n} X_{n-1}
  \xrightarrow{d_{n-1}} \ldots \xrightarrow{d_1} X_0 \rightarrow \F \rightarrow 0.
\end{equation}

Let $\F [x]$ be the span of $[x]$ for $x \in \X_{(k)}$.
The stabiliser $G_x$ acts on $\F[x]$ by
$$
\rho: G_x \to \aut_{\F}(\F [x]), \ \
\rho (\vg) = (-1)^{sign(R(\vg,x))}.
$$
Since $A$ acts trivially on $\X_\bullet$, it also acts trivially on $\F[x]$,
so $(\rho, \F[x]) \in \M_{A, 1}(G)$. 
Since $G_x$ is open and compact modulo $A$, 
$\F G \mathbin{ \underset{ \mathclap{ \F G_x} }{\otimes } }\F [x]$ is a projective object in $\M_{A, 1}(G)$
by Corollary~\ref{projective}. 
 
Let $\X_{(k)} (G)$ be a complete set of representatives of $G$-orbits
on  $\X_{(k)}$.
As a sum 
of projective objects
$\sum\limits_{x \in \X_{(k)}(G)} \F G \mathbin{ \underset{ \mathclap{ \F G_x} }{\otimes } } \F [x]$
is also projective.
We have a $G$-module isomorphism
$$
\sum\limits_{x \in \X_{(k)}(G)} \F G \mathbin{ \underset{ \mathclap{ \F G_x} }{\otimes } } \F [x]
\xrightarrow{\cong}X_k, \ \
\vg \otimes \a [x] \mapsto \a [\vg \cdot x] 
$$
so that the sequence~(\ref{exa})
is a projective resolution of $\F$ in $\M_{A, 1}(G)$.

Let $(\pi, V) \in\M_{A, \chi} (G)$. By Lemma~\ref{projres} 
\begin{alignat*}{2}
       X_{n} \otimes V \rightarrow X_{n-1} \otimes V &\rightarrow \cdots &\mathllap{\cdots}\rightarrow X_0 \tensor V \rightarrow V \rightarrow 0
          \end{alignat*}
is a projective resolution of $V$ of length at most $n$.
This concludes the proof for $\M_{A, \chi}(G)$. Since $\M_A(G)=\oplus_\chi \M_{A, \chi}(G)$,
we get $\pd(\M_{A}(G)) \leq n$. 
\end{proof}

\begin{ex}
  Let $G=\GL$, where $\K$ is a non-Archimedean local field
  and centre $Z(G)=\K^{\x}$. Let $\pi$ be a uniformizer in $\K$.
  Set $A=\< \pi^n \>$ as our closed central subgroup.
  As observed by Bernstein \cite[Th. 29]{Bern},
  the action of $G$ on its Bruhat-Tits building implies that $\pd (\M (\PGL)) \leq n$.
  Our Theorem~\ref{mainthm}
  gives not only this result but also a subtler result that 
  $\pd (\M_A (\GL)) \leq n$. 
\end{ex}

\section{Cosheaves} \label{two}

While we follow Gelfand and Manin \cite[Ch. 1]{GeMa}
with all notation and terminology, 
we choose to use the terms 
{\em a sheaf} for a cohomological coefficient system
and
{\em a cosheaf} for a homological coefficient system.
By default all our sheaves and cosheaves are with coefficients
in $\F$-vector spaces.
Our change of terminology is justified not only by its brevity:
a sheaf $\sF$ on $\sX_\bullet$ (cf. Definition~\ref{dfn_sheaf})
determines a constructible sheaf
$|\sF|$ on the geometric realisation $|\sX|$. 
The canonical bijection~(\ref{can_bi})
permits an explicit description
of the stalk $|\sF|$
at a point $p\in |\sX|$:
$$
|\sF|_{(p)} = \sF_x
\ \mbox{ where } \ 
\mathring{\tau}(\alpha,x) = p,
$$
while the restrictions are determined by the linear
structure maps $\sF(f,x): \sF_{\sX(f)x}\rightarrow \sF_x$, where $f:[m] \to [n]$.
Similarly, a cosheaf $\sC$ on $\sX_\bullet$ defines the 
constructible cosheaf
$|\sC|$ on $|\sX|$. 

Now we go back to $G$ acting continuously on $\sX_\bullet$ and $|\sX|$
with the central subgroup $A$ acting trivially.
The continuity means 
that the stabiliser $G_x$ of any simplex $x\in \sX_n$ is open in $G$.

\begin{defn}
An \emph{equivariant cosheaf} is a cosheaf $\sC$ 
with an additional data:
a linear map $\vg_x =\vg(\sC)_x: \sC_x \to \sC_{\vg x}$ 
for any $\vg \in G$ and any simplex $x$.
This data satisfies three axioms:
\begin{itemize}
\item[(i)] $\vg_{\vh x} \circ \vh_x = (\vg\vh)_x$ for any $\vg,
  \vh \in G$ and a simplex $x$.
\item[(ii)] 
$\sC_x$ is a smooth representation of $G_x$ for any simplex $x$.
\vspace{4pt}
\item[(iii)]
The square 
$
\begin{CD}
\sC_x @>>\vg_x> \sC_{\vg x}\\
@VV{\sC (f,x)}V @VV{\sC (\vg f,  \vg x)}V\\
\sC_{\sX(f)x} @>{\vg_{\sX(f)x}}>> \sC_{\sX(\vg f)\vg x}
\end{CD}
$ 
\hspace{14pt}
is commutative
for all $\vg \in G$,
\vskip 4pt
\noindent
simplices $x\in \sX_n$ and 
nondecreasing maps
$f:[m] \rightarrow [n]$.
\end{itemize}
A \emph{morphism} $\psi : \sC \rightarrow \sD$ of equivariant cosheaves
is a system of linear maps \linebreak
 $\psi_x : \sC_x \rightarrow \sD_x$,
commuting with actions
and corestrictions, i.e, 
the squares 
$$
\begin{CD}
\sC_x @>>\psi_x> \sD_{x}\\
@VV{\sC (f,x)}V @VV{\sD (f, x)}V\\
\sC_{\sX(f)x} @>{\psi_{\sX(f)x}}>> \sD_{\sX(f)x}
\end{CD}
\hspace{34pt}
\mbox{ and }
\hspace{24pt}
\begin{CD}
\sC_x @>>\psi_x> \sD_{x}\\
@VV{\vg(\sC)_x}V @VV{\vg (\sD)_x}V\\
\sC_{\vg x} @>{\psi_{\vg x}}>> \sD_{\vg x}
\end{CD}
$$ 
are commutative
for all $\vg \in G$, 
$x\in \sX_n$ and 
nondecreasing maps
$f:[m] \rightarrow [n]$.
\end{defn}

We denote the category of equivariant cosheaves by
$\Ch$.
It is an abelian category \cite{SS2}: kernels and cokernels
can be computed simplexwise. Another abelian category of interest
is the category $\Sh$ of equivariant sheaves. For the sake of completeness
we give its full definition.

\begin{defn}
  \label{dfn_sheaf}
An \emph{equivariant sheaf} is a sheaf $\sF$ 
with an additional data:
a linear map $\vg_x =\vg(\sF)_x: \sF_x \to \sF_{\vg x}$ 
for any $\vg \in G$ and any simplex $x$.
This data satisfies three axioms:
\begin{itemize}
\item[(i)] $\vg_{\vh x} \circ \vh_x = (\vg\vh)_x$ for any $\vg,
  \vh \in G$ and a simplex $x$. 
\item[(ii)] 
$\sF_x$ is a smooth representation of $G_x$ for any simplex $x$.
\vspace{4pt}
\item[(iii)]
The square 
$
\begin{CD}
\sF_x @>>\vg_x> \sF_{\vg x}\\
@AA{\sF (f,x)}A @AA{\sF (\vg f,  \vg x)}A\\
\sF_{\sX(f)x} @>{\vg_{\sX(f)x}}>> \sF_{\sX(\vg f)\vg x}
\end{CD}
$ 
\hspace{14pt}
is commutative
for all $\vg \in G$,
\vskip 4pt
\noindent
simplices $x\in \sX_n$ and 
nondecreasing maps
$f:[m] \rightarrow [n]$.
\end{itemize}
A \emph{morphism} $\psi : \sF \rightarrow \sE$ of equivariant sheaves
is a system of linear maps \linebreak
$\psi_x : \sF_x \rightarrow \sE_x$,
commuting with actions
and restrictions, i.e, 
the squares 
$$
\begin{CD}
\sF_x @>>\psi_x> \sE_{x}\\
@AA{\sF (f,x)}A @AA{\sE (f, x)}A\\
\sF_{\sX(f)x} @>{\psi_{\sX(f)x}}>> \sE_{\sX(f)x}
\end{CD}
\hspace{34pt}
\mbox{ and }
\hspace{24pt}
\begin{CD}
\sF_x @>>\psi_x> \sE_{x}\\
@VV{\vg(\sF)_x}V @VV{\vg (\sE)_x}V\\
\sF_{\vg x} @>{\psi_{\vg x}}>> \sE_{\vg x}
\end{CD}
$$ 
are commutative
for all $\vg \in G$, 
$x\in \sX_n$ and 
nondecreasing maps
$f:[m] \rightarrow [n]$.
\end{defn}

We say that an equivariant cosheaf $\sC$ (sheaf $\sF$)
is \emph{discrete}
if the stabiliser $G_x$
of any simplex $x$ acts on $\sC_x$ 
(correspondingly $\sF_x$) 
through a discrete quotient,
i.e. the kernel of this representation is an open subgroup of $G_x$.
The full subcategories of discrete equivariant cosheaves
$\dCh$ or discrete equivariant sheaves
$\dSh$ are abelian categories.

Other full subcategories are $A$-semisimple (co)sheaves, i.e.,
those (co)sheaves where each $\sF_x$ (correspondingly $\sC_x$)
is $A$-semisimple. There is a further version of 
$A$-semisimple (co)sheaves with a fixed character $\chi$.
Hence, we have six categories of equivariant cosheaves (and similarly sheaves):
$$
\begin{CD}
\Ch @<<\supseteq< \ChA@<<\supseteq< \ChAc  \\
@AA{\vsub}A @AA{\vsub}A @ AA{\vsub}A \\
\dCh @<\supseteq<< \dChA@<\supseteq<< \dChAc  
\end{CD}
$$

If $(\rho, V)$ is a smooth representation of $G$, we can associate
\emph{the trivial cosheaf} $\Vt$ 
and
\emph{the trivial sheaf} $\Vl$ 
to it.
We define 
$$
\Vl_{\, x}=\Vt_{\, x}\coloneqq V,
\ \ \ 
\Vt(f,x)\coloneqq \mbox{Id}_V ,
\ \ \ 
\Vl(f,x)\coloneqq \mbox{Id}_V ,
\ \ \ 
\vg_x\coloneqq \rho(\vg)
$$
for all $\vg \in G$, 
$x\in \sX_n$ and 
nondecreasing maps
$f:[m] \rightarrow [n]$.
The trivial cosheaf  $\Vt$
is discrete ($A$-semisimple) if and only if 
$V$ is discrete ($A$-semisimple)
if and only if 
the trivial sheaf  $\Vl$
is discrete ($A$-semisimple).

We need to work a bit harder 
to construct more interesting discrete sheaves and cosheaves. 
With this aim in mind we propose the following definition.

\begin{defn} A \emph{system of subgroups} $\sG$ of $G$ acting on $\sX_\bullet$
is a datum assigning a subgroup $\sG_x$ of the simplex stabiliser
$G_x$ to each simplex $x\in \sX_n$. The datum needs to be $G$-equivariant, i.e.,
$\vg \sG_x \vg^{-1} = \sG_{\vg x}$ for all $\vg\in G$ and $x\in \sX_n$. 
The following adjectives
will be applied to a system of subgroups $\sG$: 
\begin{itemize}
\item The system is \emph{open} if $\sG_x$ is open in $G_x$ for all $x$.
\item The system is \emph{closed} if $\sG_x$ is closed in $G_x$ for all $x$.
\item The system is \emph{cofinite} if 
the index of $\sG_x$ in $G_x$ is finite for all $x$.
\item The system is \emph{compact modulo $A$} if 
$\sG_x$ is compact modulo $A$ for all $x$.
\item The system is \emph{contravariant} if $\sG_{\sX(f)x}\subseteq\sG_x$ 
for all 
$x\in \sX_n$ and 
nondecreasing maps
$f:[m] \rightarrow [n]$.
\item The system is \emph{covariant} if $\sG_{\sX(f)x}\supseteq\sG_x$ 
for all 
$x\in \sX_n$ and 
nondecreasing maps
$f:[m] \rightarrow [n]$.
\end{itemize}
\end{defn}

Observe that the $G$-equivariance implies that
$\sG_x$ is a normal subgroup of $G_x$.
We have a minor moral dilemma which system we should call covariant
and which contravariant. We resolve this dilemma
by calling covariant
the system of stabilisers $\sG_x\coloneqq G_x$ for a label-preserving action of $G$
on a building.
We can construct interesting sheaves and cosheaves by taking invariants
and coinvariants with respect to a system of subgroups.

\begin{prop}\label{cosheaf}
Let $\sG$ be a system of subgroups and $(\rho, V)$ a smooth $G$-representation. 
The following statements hold:
\begin{enumerate}
\item If $\sG$ is contravariant, then
the invariants $\VtI_{\, x}\coloneqq V^{\sG_x}$ 
is an equivariant cosheaf
and
the coinvariants $\VlC_{\, x}\coloneqq V_{\sG_x}$ 
is an equivariant sheaf.
\item If $\sG$ is covariant, then
the invariants $\VlI_{\, x}\coloneqq V^{\sG_x}$ 
is an equivariant sheaf
and
the coinvariants $\VtC_{\, x}\coloneqq V_{\sG_x}$ 
is an equivariant cosheaf.
\item If, further to (1) or (2), $\sG$ is open,
then the (co)sheaf is discrete.
\item If, further to (1) or (2), $V$ is 
$A$-semisimple (with a fixed character $\chi$), 
  then the sheaves $\VlI$, $\VlC$ and
  the cosheaves $\VtI$, $\VtC$
are 
$A$-semisimple (with a fixed character $\chi$ correspondingly). 
\end{enumerate}
\end{prop}
\begin{proof}
One of the invariant spaces $V^{\sG_x}$ and $V^{\sG_{\sX(f)x}}$ 
contains the other one. Which contains which depends on whether
the system of subgroups is contravariant or covariant.
More precisely,
a covariant system produces a sheaf,
while a contravariant system produces a cosheaf.
The action of $G$ is given by $\rho$ in both cases: $\vg_x\coloneqq \rho(x)$.

The coinvariant spaces $V_{\sG_x}$ and $V_{\sG_{\sX(f)x}}$
are connected by a natural surjection.
Similarly to invariants,
a contravariant system produces a sheaf,
while a covariant system produces a cosheaf.
The action of $G$ is again given by $\rho$. 

The last two statements are immediate.
\end{proof}

Cosheaves appear more suitable than sheaves for studying representations
in this simplicial environment. 
We turn our attention to cosheaves, 
commenting later on difficulties one faces with sheaves.
The simplicial homology complex of $\sX$
with coefficients in $\sC$ is defined similarly to 
Equation~(\ref{boundary_s3})
(cf. \cite{GeMa}):
$$
C^\sharp_n (\sX_\bullet , \sC)\coloneqq \Big \{ \sum_{x\in \sX_n} \alpha_x x \  
\Big | \ 
\alpha_x\in \sC_x,  
\mbox{ all but finitely many }
\alpha_x =0
\Big \} ,
$$
$$
d_0 \coloneqq 0, \ \ 
d_n \Big(\sum_{x\in \sX_n} \alpha_x x \Big) \coloneqq 
 \sum_{x\in \sX_n} \sum_{i=0}^n
(-1)^i
[\sC (\partial^i_n, x) (\alpha_x)] 
[\sX(\partial^i_n)(x)]
$$
for $n>0$.
Since degenerate simplices span a subcomplex
$(C_n^\flat (\sX_{(\bullet)} , \sC), d_n )$,
our key complex is the quotient complex
$$
C_k(\sX_{\bullet} , \sC) \coloneqq
C^\sharp_k(\sX_{\bullet} , \sC) / C^\flat_k (\sX_{\bullet} , \sC)
$$
spanned by linear combinations 
of non-degenerate simplices $\sum_{x\in \sX_{(n)}}\alpha_x [x]$. 
For an open subgroup $K\leq G$ we introduce 
the full subcategory $\sM(G)^K$
 of $\M(G)$
whose objects are smooth $G$-representations 
generated by their $K$-fixed vectors.
Also, let $\sM(G)^\circ$ be the union of various
$\sM(G)^K$. Its objects are those smooth representations
that are generated by $K$-fixed vectors
for some open subgroup $K \sset G$. Inside them we 
have the corresponding $A$-semisimple categories
$$
\M_A(G)^K, \; 
\M_A(G)^\circ , \; 
\M_{A, \chi}(G)^K \ 
\mbox{ and } \  
\M_{A, \chi}(G)^\circ .
$$

\begin{prop}
\label{chain_properties}
Let $\sC$ be a $G$-equivariant cosheaf on $\sX_\bullet$.
Let $x_1, x_2 \ldots$ be representatives of $G$-orbits
on $\sX_{(n)}$.
Then the  following statements hold: 
\begin{enumerate}
\item Chains $C_n (\sX_{\bullet} , \sC)$ 
and homologies $H_n (\sX_{\bullet} , \sC)$ are smooth $G$-representations.
\item There is an isomorphism of $G$-modules
$$
C_n (\sX_{\bullet} , \sC)
\cong
\bigoplus_k a-\ind_{G_{x_k}}^G \sC_{x_k}.
$$
\item If  $\sC$ is $A$-semisimple (with a character $\chi$),
then chains $C_n (\sX_{\bullet} , \sC)$ 
and homologies $H_n (\sX_{\bullet} , \sC)$
are $A$-semisimple (with a character $\chi$ respectively).
\item If  $\sC$ is discrete  and $\sX_{(n)}$ has finitely many
  $G$-orbits, then
chains $C_n (\sX_{\bullet} , \sC)$ and
homologies $H_n (\sX_{\bullet} , \sC)$ 
are in $\sM(G)^\circ$.
More precisely,  
$C_n (\sX_{\bullet} , \sC)$ and
$H_n (\sX_{\bullet} , \sC)$ 
are in $\sM(G)^K$
where $K=K_1\cap K_2 \cap \ldots \cap K_k$ and 
$K_i$ is the kernel of
the $G_{x_i}$-representation $\sC_{x_i}$.
\item If $\sX_{(n)}$ has finitely many $G$-orbits
and $\sC_{x_k}$ is finitely generated $G_{x_k}$-module for each $x_k$, 
then chains $C_n (\sX_{\bullet} , \sC)$ and
homologies $H_n (\sX_{\bullet} , \sC)$ 
are finitely generated $G$-modules. 
\item Suppose that for each $x\in \sX_n$,
  the stabiliser $G_x$ is compact
  modulo $A$ and the field $\F$ is $G_x/A$-ordinary. 
If  $\sC$ is $A$-semisimple (with a character $\chi$),
then
the space of chains $C_n (\sX_{\bullet} , \sC)$ is a projective 
object in $\M_A(G)$ (correspondingly in $\M_{A,\chi}(G)$). 
\end{enumerate}
\end{prop}
\begin{proof}
  The $G$-action on the chains is defined as in
Equation~(\ref{action_s3}):
$$
\vg \cdot \Big(\sum_{x\in \sX_{(n)}} \alpha_x [x] \Big) \coloneqq 
 \sum_{x\in \sX_{(n)}} 
(-1)^{sign(R(\vg,x))}
\alpha_x \, [\vg \cdot x] \, .
$$
All statements are proved one by one from (1) to (6). 
Statement (6) requires
Corollary~\ref{projective}, while
the rest of the statements are
straightforward. 
\end{proof}

Let us examine the functors connecting cosheaves and representations.
The functors from representations to cosheaves are {\em localisation functors}: they produce an equivariant cosheaf, a local object from a representation. The easiest localisation functor is the trivial cosheaf:
$$
\sL: \sM(G) \rightarrow \Ch , \ \ 
\sL((\rho, V)) = \Vt.
$$
In the opposite direction, we have \emph{homology functors}
$$
\sH: \Ch \rightarrow \sM(G) , \ \ 
\sH(\sC) = H_0 (\sX_\bullet, \sC). 
$$
Let
$\Sigma \subset \mbox{Mor}(\Ch)$ 
be the class of those morphisms $f$ such that $\sH(f)$ is an isomorphism.
We get a functor from the category of left fractions
\cite[I.1.1]{GZ}:
$$
\sH[\Sigma^{-1}]: \Ch[\Sigma^{-1}] \rightarrow \sM(G).
$$ 
The category of fractions always exists and admits a natural fraction functor \linebreak
$\sQ_{\Sigma}: \Ch\rightarrow \Ch[\Sigma^{-1}]$.
However, in general this category is intractable. It needs to satisfy the left Ore conditions (or admit the left calculus of fractions in the terminology of Gabriel and Zisman \cite[I.2.2]{GZ}) to enable working with them:
\begin{lemma} \cite[I.3]{GZ}
\label{colimits}
Let $\mA$ be an abelian category, $\Sigma$ a class of morphisms
in it admitting a left calculus of fractions.
Then $\mA[\Sigma^{-1}]$ is an additive category with finite colimits.
\end{lemma}
In particular, there are cokernels in $\mA[\Sigma^{-1}]$.
An instructive exercise is to show that for a morphism $s^{-1} f$ in $\mA[\Sigma^{-1}]$ the composition $\mbox{coker}(f)s$ is its cokernel, yet
$\mbox{ker}(f)$ is not necessarily its kernel. To obtain a kernel one needs the right calculus of fractions. If $\Sigma$ admits both left and right calculi of
fractions, then $\mA[\Sigma^{-1}]$ is abelian \cite[I.3.6]{GZ}.

We are ready for the main theorem of the section, which is a
generalisation of Localisation Theorem by Schneider and Stuhler \cite[Theorem V.1]{SS2}.
We follow their strategy in our proof. It is important to notice that
no restriction on $\F$ appears in the theorem.

\begin{thm} (Localisation Theorem) 
\label{local}
Consider a continuous action of the locally compact totally
disconnected group $G$ on a simplicial set $\sX_\bullet$,
where the central subgroup $A$ acts trivially. 
The following statements hold.
\begin{enumerate}
\item The class $\Sigma$ of morphisms $f$ in $\Ch$
such that $\sH (f)$ is an isomorphism
admits a calculus of left fractions.
\item 
$\sH[\Sigma^{-1}]: \Ch[\Sigma^{-1}] \rightarrow \sM(G)$
is conservative, i.e., a morphism $f$ is an isomorphism if and only if
$\sH[\Sigma^{-1}](f)$ is an isomorphism.
\item 
$\sH[\Sigma^{-1}]$ commutes with colimits.
\item 
$\sH[\Sigma^{-1}]$ is faithful, i.e., injective on morphisms.
\end{enumerate}
If $|\sX|$ is connected, then the following three statements hold:
\begin{enumerate}
\item[(5)] 
$\sH[\Sigma^{-1}]: \Ch[\Sigma^{-1}] \rightarrow \sM(G)$
is an equivalence of categories. 
\item[(6)] $\sQ_{\Sigma}\circ \sL$ is a quasi-inverse of
  $\sH[\Sigma^{-1}]$.
\item[(7)] These equivalences restrict to equivalences
$\ChA[\Sigma_A^{-1}] \xrightarrow{\cong} \sM_A(G)$
and
$\ChAc[\Sigma_{A,\chi}^{-1}] \xrightarrow{\cong} \sM_{A,\chi}(G)$
where $\Sigma_{A}$ and $\Sigma_{A,\chi}$
are intersections of $\Sigma$ with the corresponding subcategories.
\end{enumerate}
\end{thm}
\begin{proof}
A short exact sequence of cosheaves gives rise to a long exact
sequence in homology. Consequently, the functor $\sH$ is right exact. Hence, 
it commutes with finite direct limits (cf. \cite[Prop. 3.3.3]{KaS}, the statement proved there is that a left exact functor commutes with finite inverse limits. Apply the
opposite categories to dualise it). The first three statements follow \cite[I.3.4]{GZ}.

Suppose $\sH[\Sigma^{-1}](f)=\sH[\Sigma^{-1}](f^\prime)$
for two morphisms $f$ and $f^\prime$. To prove that $f=f^\prime$ it suffices to show that $\mbox{coker}(f-f^\prime)$ is an isomorphism
(cokernels exist by Lemma~\ref{colimits}).
By (3), $\sH[\Sigma^{-1}](\mbox{coker}(f-f^\prime))=
\mbox{coker}(\sH[\Sigma^{-1}](f)-\sH[\Sigma^{-1}](f^\prime))
= \mbox{coker}(0)$ is an isomorphism.
By (2) $\mbox{coker}(f-f^\prime)$ is an isomorphism.
This proves (4).

Since $|\sX|$ is connected, we have an exact sequence
$$
C_1 (\sX_\bullet , \F) \xrightarrow{d_1}
C_0 (\sX_\bullet , \F) \xrightarrow{w} \F \rightarrow 0, \ \ 
w\Big(\sum_x \alpha_x x \Big) = \sum_x \alpha_x.
$$
Observe that for a smooth $G$-representation $V$
the tensor product $C_k (\sX_\bullet , \F)\otimes V$ is naturally isomorphic as a $G$-representation to $C_k (\sX_\bullet , \Vt)$. Hence, tensoring with $V$ produces
another exact sequence
$$
C_1 (\sX_\bullet , \Vt) \xrightarrow{d_1}
C_0 (\sX_\bullet , \Vt) \rightarrow V \rightarrow 0
$$
that gives a natural isomorphism
$\sH[\Sigma^{-1}] \circ (\sQ_{\Sigma} \circ \sL ) \cong
\mbox{Id}_{\sM(G)}$: 
$$
\sH[\Sigma^{-1}] (\sQ_{\Sigma} (\sL (V)))
\cong
\sH(\sL (V))
=H_0 (\sX_\bullet , \Vt) \xrightarrow{\cong}
V.
$$
In the opposite direction, we need a natural transformation
$$
\gamma: \mbox{Id}_{\Ch[\Sigma^{-1}]} \rightarrow 
(\sQ_{\Sigma} \circ \sL ) \circ \sH[\Sigma^{-1}]
$$
that we define in $\Ch$ for each cosheaf $\sC$ by
$$
\gamma (\sC)_x \coloneqq 
\left\{ \begin{array}{ccc} 
\sC_x\ni\alpha \mapsto 0 \in \sH(\sC) & \mbox{ if } & x\in \sX_n,
\ n>0\, ,
\\ 
\sC_x\ni\alpha \mapsto [\alpha x] \in \sH(\sC) & \mbox{ if } & x\in
\sX_0\, .
\end{array}\right.
$$
Observe that $\sH (\gamma (\sC))$
is an isomorphism. By (2), $\gamma (\sC)$
is an isomorphism, so $\gamma$ is a natural isomorphism.
This proves (5) and (6).

To attack (7), observe a fine difference between
$\ChA[\Sigma^{-1}]$
and $\ChA[\Sigma_A^{-1}]$. The former is a full subcategory of
$\Ch[\Sigma^{-1}]$, 
while the latter is the category of fractions of $\ChA$. 
They are connected by a natural functor \linebreak
$\sN: \ChA[\Sigma_A^{-1}] \rightarrow\ChA[\Sigma^{-1}]$,
identical on objects and morphisms. Clearly, $\sN$ is an equivalence.
It remains to observe $\sH (\ChA[\Sigma^{-1}])\subseteq \sM_A(G)$
and $\sQ_{\Sigma} (\sL(\sM_A(G)))\subseteq \ChA[\Sigma^{-1}]$.
Both inclusions are straightforward.
\end{proof}

Theorem~\ref{local} may or may not bring
any new information about representations of $G$ to the table.
For instance, any $G$ acts on the point. Then this theorem is a
tautology, producing the identity functor on $\sM(G)$. 
Another interesting thought experiment is to replace $G$ with a
product $G\times H$ 
where $H$ acts trivially on $\sX_\bullet$. All information about the
$H$-action 
in $\sM(G\times H)$
is swiped under the carpet in 
${{\mathrm C}{\mathrm s}{\mathrm h}_{G\times H}(\sX_\bullet)}$: 
$H$ needs to act somehow on all $\sC_x$ for all equivariant
cosheaves. 
On the other hand, Theorem~\ref{mainthm} demonstrates that
the localisation over simplicial sets can provide new non-trivial information. 

Can we trim down the category of cosheaves by using systems of
subgroups? If $\sG_x$ is a contravariant system of subgroups,
we have an exact sequence
\begin{equation}
\label{seq1}
C_1 (\sX_\bullet , \VtI) \xrightarrow{d_1}
C_0 (\sX_\bullet , \VtI) \xrightarrow{w} V, \ \ 
w\Big(\sum_x \alpha_x x \Big) = \sum_x \alpha_x.
\end{equation}
Using it, we can get a version of Theorem~\ref{local} for discrete cosheaves.
Let $\Sigma^{\circ}$, $\Sigma_A^{\circ}$ and $\Sigma_{A,\chi}^{\circ}$ be
the intersections of $\Sigma$
with $\dCh$, $\dChA$ and $\dChAc$ correspondingly.
\begin{cor}
  Suppose that $|\sX|$ is connected
  and there are finitely many $G$-orbits
on $\sX_0$. 
Suppose further that for any representation $V\in \sM(G)^\circ$
there exists an open contravariant system
of subgroups $\sG$ such that the following variation of sequence~{(\ref{seq1})} is exact:
$$
C_1 (\sX_\bullet , \VtI) \xrightarrow{d_1}
C_0 (\sX_\bullet , \VtI) \xrightarrow{w} V
\rightarrow 0 
.$$
Then the functor $\sH[\Sigma^{-1}]$
provides equivalences
$
\dCh[\Sigma^{\circ\; -1}] \xrightarrow{\cong} \sM(G)^\circ$,
$$
\dChA[\Sigma_A^{\circ\;-1}] \xrightarrow{\cong} \sM_A(G)^\circ, 
\ \mbox{ and } \ 
\dChAc[\Sigma_{A,\chi}^{\circ\;-1}] \xrightarrow{\cong} \sM_{A,\chi}(G)^\circ.
$$
\end{cor}
\begin{proof}
  Similarly to the proof of Theorem~\ref{local},
the difference between $\dCh[\Sigma^{-1}]$
and $\dCh[\Sigma^{\circ\,-1}]$ is immaterial.
Parts (3) and (4) of Proposition~\ref{chain_properties}
tell us that
$\sH[\Sigma^{-1}]$
is a well-defined functor
$\dCh[\Sigma^{\circ\; -1}] \rightarrow \sM(G)^\circ$.
Ditto for the $A$-semisimple categories.

If $V\in\sM(G)^\circ$, we pick the aforementioned (in the statement) 
system of subgroups $\sG$.
Then the trivial cosheaf
$\sQ_{\Sigma} (\sL(V))=\Vt$
is isomorphic to the cosheaf $\VtI$ in $\Ch[\Sigma^{-1}]$.
The latter cosheaf is discrete because
$G_x$ acts on $\VtI_x$ via the discrete quotient
$G_x/\sG_x$.
It is easy to see $A$-semisimplicity through as well.
\end{proof}

If $V$ is admissible and $\sG$ is compact open,
then
the cosheaf $\VtI$ is finite dimensional, i.e.,
each vector space $\VtI_{\, x}$ is finite dimensional.
Thus, it has a chance of giving us a resolution
of $V$ by finitely generated projective modules.
We will address this problem in the next section.

Finally, let us comment why we think
cosheaves are better than sheaves
for our studies.
If $\sF$ is an equivariant sheaf,
the cohomology
$C^n(\sX,\sF)$ is not necessarily a smooth
representation of $G$.
Taking its smooth part, one gets
a smooth cohomology complex
$C_{sm}^\bullet (\sX,\sF)$, whose relation
to the topology of $\sX$ is more remote
than of the original complex
$C^\bullet (\sX,\sF)$.
In particular, one could expect a subtle, yet fruitful
interplay between
$C^\bullet(\sX,\Vl)$, $C_{sm}^\bullet(\sX,\Vl)$ and $V$,
but it remains to be seen whether this mesh is
capable of producing something useful, for instance,
injective resolutions of $V$.

\section{Schneider-Stuhler Resolution} \label{fgres}

We call a finitely generated projective resolution
of the form $C_\bullet(\sX_\bullet ,\VtI)$
{\em a Schneider-Stuhler resolution}, acknowledging their construction
for $p$-adic reductive groups \cite{SS2}.
Where do suitable (for such resolutions) systems of subgroups come from?

Denote by 
$f_i^n: [0] \rightarrow [n]$
the function $f_i^n(0)=i$. 
Suppose we are given a compact open subgroup $\sG_x$ for each vertex $x\in \sX_0$
such that
\begin{enumerate}
\item $\sG_{\vg x} =\vg \sG_x \vg^{-1}$ for all $\vg\in G$, $x\in\sX_0$ and
\item $\sG_x \sG_y = \sG_y \sG_x$ if $x$ and $y$ are adjacent,
  i.e., $x=\sX (f_0^1)(w),\;y=\sX (f_1^1)(w)$ for some $w\in\sX_1$.
\end{enumerate}
Condition (2) allows us to extend this collection of subgroups
to a compact open contravariant system of subgroups by taking products over
vertices:
$$
\sG_x \coloneqq \sG_{\sX(f_0^n)x} \sG_{\sX(f_1^n)x} \cdots \sG_{\sX(f_n^n)x}
\ \mbox{ for all } \ x\in\sX_n.
$$
We call a compact open contravariant system obtained by this construction
from some initial choice of subgroups {\em an exquisite system}.

If the field $\F$ is $\sG_x$-ordinary for each $x\in\sX_0$,
then it is $\sG_x$-ordinary for each $x\in\sX_\bullet$ as soon as we deal
with an exquisite system.  As observed by Meyer and Solleveld \cite{MeSo}, 
this gives us idempotents
$\L_x \coloneqq \L_{\sG_x} \in \sH= \sH (G, \F, \mu)$
for a suitable choice of $\mu$ (notation of Section~\ref{zero}).
Not only are these idempotents convenient for calculations
but also they control the invariants:
$V^{\sG_x}=\L_x\ast V$. The following lemma is proved for
Bruhat-Tits buildings of $p$-adic groups 
by Meyer and Solleveld \cite[Lemma 2.6]{MeSo}: 
\begin{lemma} \label{prop_id}
The collection $\L_x,\; x \in \sX_\bullet$ of idempotents
arisen from an exquisite system of subgroups
satisfies the following identities:
\begin{enumerate}
\item $\L_x\star \L_y = \L_y \star \L_x$ if $x,y\in\sX_0$ are adjacent.
\item $\L_x = \L_{\sX(f_0^n)x} \star \L_{\sX(f_1^n)x} \star \ldots \star \L_{\sX(f_n^n)x} \ $
for all  $x\in\sX_n$.
\item $\L_{\vg \cdot x}= \,^\vg{\L_x}^{\vg^{-1}}$ for all $\vg \in G$, $\vx\in\sX_\bullet$.
\end{enumerate}
\end{lemma}
\begin{proof}
By definition 
\begin{equation} 
\label{comm_id}
\L_x\star \L_y (\vg) = \int_G \L_x(\vh) \L_y(\vh^{-1}\vg) \mu (d \vh).
\end{equation}
The integrand vanishes unless $\vh\in \sG_x,\; \vh^{-1}\vg\in\sG_y$. 
Thus  $\L_x\star \L_y$ is supported on $\sG_x \sG_y$.
Moreover,
$\vh^{-1}\vg\in\sG_y$ translates into 
$\vh\in\vg \sG_y$ so that
(\ref{comm_id}) becomes
\begin{equation} 
\label{comm_id2}
\int_{\sG_x\cap\vg\sG_y} \L_x(\vh) \L_y(\vh^{-1}\vg) \mu (d \vh)
=
\frac{\mu(\sG_x\cap\vg\sG_y)}{\mu(\sG_x)\mu(\sG_y)}.
\end{equation}
Decomposing $\vg=\vh (\vh^{-1}\vg)$ for some $\vh\in \sG_x,\; \vh^{-1}\vg\in\sG_y$,
(\ref{comm_id2}) becomes
$$
\frac{\mu(\sG_x\cap\vh\sG_y)}{\mu(\sG_x)\mu(\sG_y)}
=
\frac{\mu(\vh^{-1}(\sG_x\cap\vh\sG_y))}{\mu(\sG_x)\mu(\sG_y)}
=
\frac{\mu(\sG_x\cap\sG_y)}{\mu(\sG_x)\mu(\sG_y)}
=
\frac{1}{|\sG_x : (\sG_x\cap\sG_y)|\mu(\sG_y)}
=
$$
$$
\frac{1}{|\sG_x\sG_y:\sG_y|\mu(\sG_y)}
=
\frac{1}{\mu(\sG_x\sG_y)}
= \L_{\sG_x\sG_y} (\vg).
$$
Since $\sG_x\sG_y=\sG_y\sG_x$, we have proved not only (1) but a stronger
equation
\begin{equation} 
  \label{comm_id3}
  \L_x\star \L_y= \L_{\sG_x\sG_y}=\L_y\star \L_x.
\end{equation}
Statement (2) follows from Equation~(\ref{comm_id3})
by an easy induction and the last statement is obvious.
%
%
%
%
\end{proof}

Let $|\X|$ be the geometric realisation of the simplicial set $\sX_\bullet=(\X_n)$. 
For a non-degenerate $x \in \sX_{(n)}$
we denote the corresponding simplex in $|\sX|$ by
$\mathring{\D}_n\times x$ 
and its points by
${\bf{x}} = (\alpha, x)$,
${\bf{y}} = (\alpha, y)$, etc. 
A particular point of interest is the centre
$\mathring{x}= ((\frac{1}{n+1},\ldots \frac{1}{n+1}),x)$ 
(see Section~\ref{two}).

We make an additional assumption that $|\sX|$ admits a CAT(0)-metric. 
Then $|\sX|$ is a unique geodesic space \cite{Bri},
in particular, any two points ${\bf{x}}, {\bf{y}} \in |\sX|$
can be connected by a unique geodesic, which we denote by
$[{\bf{x}},{\bf{y}}]$.
A subset $Y\sset |\sX|$ is called {\em convex}
if $[{\bf{x}},{\bf{y}}]\sset Y$
for all ${\bf{x}},{\bf{y}}\in Y$.
{\em The convex hull $\mH(Y)$ of $Y$} is the intersection of all convex subsets
of $|\sX|$ containing $Y$.
Notice that $[{\bf{x}},{\bf{y}}]=\mH (\{ {\bf{x}},{\bf{y}} \} )$.

Let $\sG$ be a system of subgroups of $G$.
We would like to have some control over the subgroups $\sG_x$, along geodesics.
Bearing this in mind, we propose the following definition:
\begin{defn}
\label{geodesic}  
  We say that a contravariant system of subgroups $\sG$ is \emph{geodesic}
  if for all ${\bf{x}},{\bf{y}}\in |\sX|$ 
  $$ \sG_z \sset \sG_x \sG_y$$
  where $z\in \sX_0$ is a vertex of the first simplex 
  $u\in \sX_n$ along the geodesic $[{\bf{x}},{\bf{y}}]$, i.e., 
$z=\sX(f_i^n)u$ for some $i$ and 
$(\D_n\times u) \cap [{\bf{x}},{\bf{y}}] = [{\bf{x}},{\bf{v}}]$
  for some ${\bf v} \in \, ]{\bf{x}},{\bf{y}}]$.
\end{defn}
The significance of this definition transpires in the following lemma,
inspired by similar results
of Meyer and Solleveld for 
Bruhat-Tits buildings of $p$-adic groups:  
\begin{lemma} (cf. \cite[Prop 2.2 and Lemma 2.6]{MeSo}) \label{prop2_id}
  Suppose that $|\sX|$ admits a CAT(0)-metric, $\sG_x$ is
  a geodesic exquisite system and the field $\F$ is $\sG_x$-ordinary
  for each $x\in \sX_0$. Then
  $$
  \L_x \star\L_z \star\L_y = \L_x \star\L_y
  \ \mbox{ and } \
  \L_x \star\L_z  = \L_z \star\L_x , 
  $$
as soon as $x,y,z \in \sX_\bullet$
satisfy the conditions spelled out in Definition~\ref{geodesic}.  
\end{lemma}
\begin{proof}
If $z=\sX(f_i^n)u$ as in Definition~\ref{geodesic},
then $\L_x$ is a product of various $\L_{\sX(f_k^n)u}$,
hence, commutes with $\L_z$. The first equality easily follows
from the geodesic condition $ \sG_z \sset \sG_x \sG_y$.
\end{proof}



Now consider 
a character $\chi: A \to \widetilde{\F}^\x$. 
Given a subgroup $H\leq G$, set $H_{\chi}\coloneqq H/H\cap \ker(\chi)$. 
It is a subgroup of $G_\chi$. 
Observe that $H_\chi$ is compact if and only if $H$ is compact modulo $A$.
We are ready for
the main
conjecture of this section:

\begin{conj}
  \label{cat0thm}
Let $G$ be a locally compact totally disconnected group, $A$ its closed central subgroup. 
Suppose $G$ acts smoothly on a simplicial set $\sX_\bullet$ of
dimension $n$, 
with $A$ acting trivially. 
Further suppose that a face of a non-degenerate simplex in $\sX_\bullet$ is non-degenerate
and $|\sX|$ admits a CAT(0)-metric
such that the faces are geodesic, i.e., 
$\mH (\mathring{\D}_n\times x) = \mathring{\D}_n\times x$ 
for each $x\in\sX_{(\bullet)}$. 
If $V\in\sM_{A,\chi}(G)$, the following four statements should conjecturally hold:
\begin{enumerate}
\item If $\sG$ is a geodesic exquisite 
system of subgroups
of $G_\chi$ such that $\F$ is $\sG_x$-ordinary for all $x\in \sX_0$, then the complex
$$
0 \rightarrow C_{n}(\sX_\bullet , \VtI) \xrightarrow{d_n} C_{n-1} (\sX_\bullet , \VtI) 
\xrightarrow{d_{n-1}} \ldots \xrightarrow{d_1} C_0 (\sX_\bullet , \VtI) \xrightarrow{w} V
$$ 
is an exact sequence.
\item Each $C_{k}(\sX_\bullet , \VtI)$ is a  projective module  in $\M_{A, \chi}(G)$.   
\item If $(\pi, V)$ is generated by invariants $V^{\sG_x}$ for some $x\in \sX_0$, 
then the complex is a projective resolution of $V$ in $\M_{A, \chi}(G)$.
\item If $(\pi, V)$ is admissible and $\sX_{(k)}$ has finitely many $G$-orbits,
then $C_{k}(\sX_\bullet , \VtI)$ is a finitely generated $G$-module.
\end{enumerate}
\end{conj}
In fact, statements (2)--(4) are established in  Proposition~\ref{chain_properties}.
Only statement (1) is truly a conjecture.
It is proved for
affine Bruhat-Tits buildings 
by Meyer and Solleveld \cite[Theorem 2.4]{MeSo}.  
We can prove its partial case:
\begin{thm}
  \label{tree}
  If the dimension of $|\sX |$ is one,
  then Conjecture~\ref{cat0thm} holds.
\end{thm}
\begin{proof}
{\bf (1), exactness at $C_{0}(\sX_\bullet , \VtI)$:}  
The inclusion $\im(d_1) \sset \ker(w)$ is clear. 
Let us show that $\ker(w) \sset \im(d_1)$. 
Pick a 0-cycle
$\alpha = \sum\limits_{i=1}^n \a_i x_i \in C_0(\sX_\bullet , \VtI)$
where all $\alpha_i\neq 0$.  
Consider the hull of its support $Y\coloneqq  \mH (\{\mathring{x}_1,\ldots,\mathring{x}_n\})$.
Under our conditions $|\sX|$ is a tree, so $Y$ is a finite tree.
Hence, $Y$ has an endpoint. Without loss of generality, $\mathring{x}_1$ is an endpoint.
Let $x'_1\in\sX_0$ be the unique vertex adjacent to $x_1$
such that $\mathring{x}^\prime_1\in Y$. 
Let $e_1\in\sX_1$ be the edge connecting $x_1$ and $x'_1$.
Since $w(\alpha) = \sum \limits_{i=1}^n \a_i=0$ and $\L_{x_i} (\alpha_i)=\alpha_i$,
we conclude that
\begin{equation}
  \label{idem_eq}
  \sum \limits_{i=1}^n \L_{x_i} (\a_i)=0 .
  \end{equation}
Applying $\L_{x_1}\star (1-\L_{x'_1})$ to Equation~(\ref{idem_eq}), 
we can rewrite each summand separately, using Lemmas~\ref{prop_id} and \ref{prop2_id}: 
 \begin{itemize}
 \item $\L_{x_1}\star (1-\L_{x'_1}) \star \L_{x_1}(\a_1)=(1-\L_{x'_1})(\a_1)$, 
 \item $\L_{x_i}\star (1-\L_{x'_1}) \star \L_{x_i}(\a_i)=0$ for $i\geq 2$.
 \end{itemize}
 Thus, $\a_1 \in \ker(1-\L_{x'_1})$ and  $\a_1 \in \im (\L_{x'_1})$.
Then
$$
\alpha^\prime \coloneqq
\a_1 x_1^\prime + \sum\limits_{i=2}^n \a_i x_i =
d_1(\pm \alpha e_1) + \alpha \in C_0(\sX_\bullet , \VtI)  
$$
and the hull of the support of $\a^\prime$ is a proper subset of $Y$.
An easy induction on the size of the hull of the support completes the proof.

{\bf (1), exactness at $C_{1}(\sX_\bullet , \VtI)$:}  
Pick a 1-cycle
$\alpha = \sum\limits_{i=1}^n \a_i x_i \in C_1(\sX_\bullet , \VtI)$
where all $\alpha_i\neq 0$.  
Consider the hull of its support $Y\coloneqq  \mH (\{\mathring{x}_1,\ldots,\mathring{x}_n\})$.
Again $Y$ is a finite tree, so $Y$ has an endpoint, e.g., $\mathring{x}_1$.
Let $z\in\sX_0$ be the unique vertex of the edge $x_1$
such that $\mathring{z}\not\in Y$.
Clearly, $d_1 (\alpha) = \pm \alpha_1 z + \ldots$
has a non-zero coefficient in front of $z$. This proves that $\alpha=0$ and $d_1$ is injective.
\end{proof}

\section{Davis Building for a Group with Generalised BN-pair} \label{real}


Let $G$ be an abstract group.
Following Iwahori \cite{Iw},
\emph{a generalised BN-pair} on a group $G$ is a triple $(B,N,S)$ 
satisfying  the following conditions:
\begin{itemize}
\item[(i)] $B$ and $N$ are subgroups of $G$. $H = B \cap N$ is a normal subgroup of $N$.
\item[(ii)] $N/H = \Omega \ltimes W$ where $\Omega$ is a subgroup and $W$ is a normal subgroup.
\item[(iii)] $W$ is generated by the set $S$.  
  The elements of $S$ have the following properties:
\begin{itemize}
\item[(iii.1)]  For any $\vt$ in $\Om \ltimes W$ and any $\vs \in S$ we have
  $\dot{\vt} B \dot{\vs} \subset B \dot{\vt} \dot{\vs} B \cup B \dot{\vt} B$
  where $\dot{\vt}$ and $\dot{\vs}$ are elements of $G$ lifting $\vt$ and $\vs$.
\item[(iii.2)] $\vs^2=1$ and $\dot{\vs} B \dot{\vs}^{-1} \neq B$ for all $\vs \in S$.
\end{itemize}
\item[(iv)] $\va S \va^{-1} = S$ for all $\va \in \Om$.
\item[(v)] $\dot{\va} B \dot{\va}^{-1} = B$ for all $\va \in \Om$ and $B \dot{\va} \neq B$
  for any  $\va \in \Om \setminus \{ 1 \}$.
\item[(vi)] $G$ is generated by $B$ and $N$.
\end{itemize}

As usual $W$ is called the Weyl group of $G$.
Note that $W$ is a Coxeter group and thus $(W,S)$ is a Coxeter system.
We call $\Om \ltimes W$ \emph{the generalised Weyl group}. 
It is rather ironic that a BN-pair is a triple but it is a moot point
whether $B$ and $N$ uniquely determine $S$ for generalised BN-pairs.
Thus, we include $S$ into the definition for safety.


Given a group $G$ with a generalised BN-pair,
we can find a smaller group $G_0$
inside $G$ which has a BN-pair.
More precisely, define $G_0 \coloneqq  BWB$. 
Then the following statements
hold \cite{Iw}:
\begin{lemma}
\label{BN_pro}
\begin{enumerate}
\item  $G_0$ is a normal subgroup of $G$ and $G/G_0 \cong \Omega$.
\item $(B, N_0)$ is a BN-pair for $G_0$, where $N_0 = N \cap G_0$.
The Weyl groups of $G_0$ and $G$ are the same.
\item The automorphism of $G_0$ defined by conjugation by an element $\vg \in G$
preserves the BN-pair up to conjugacy in $G_0$, i.e., 
there exists   $\vg_0 \in G_0$ such that $\vg B \vg^{-1}=\vg_0 B \vg_0^{-1}$ 
and $\vg N_0 \vg^{-1}=\vg_0 N_0 \vg_0^{-1}$.
\end{enumerate}
\end{lemma}
A group with a BN-pair is an obvious example of a group with generalised BN-pair.
For a subtler example, consider 
a group $G$ with a BN-pair
$(B,N)$ and another  group $\Omega$.
The group $\Omega\times G$ admits a BN-pair $(\Omega\times B,\Omega\times N)$
and a generalised BN-pair $(B,\Omega\times N)$. The Weyl groups are the same in both cases
but the generalised Weyl group is bigger: $N/H=\Omega\times W$ for the latter pair.

For an example pertinent for our investigation \cite{Iw},
consider $G=\GL$ over a non-Archimedean local field $\K$, its Iwahori
subgroup $I$ and its subgroup of monomial matrices $N$.
The pair $(I,N)$ is a generalised BN-pair:
$I\cap N = \Dn (\O_\K^\times, \ldots, \O_\K^\times ) \cong (\O_\K^\times)^n$
consists of diagonal matrices with coefficients
in the ring of integers $\O_\K\leq \K$.
Denote $T= \Dn (\K^\times, \ldots, \K^\times) \cong (\K^\times)^n$.
Then $N/(I\cap N) \cong N/T \ltimes T/H \cong \Sn \ltimes \Z^n$.
It contains the Weyl group
$W=\Sn \ltimes \Z_0^n$ of type $\widetilde{A}_{n-1}$
as a normal subgroup 
(where $\Z_0^n = \{ (x_i) \mid \sum_i x_i =0\}$)
with a complementary group 
$\Omega = \langle (1,0,\ldots, 0)\cdot \gamma \rangle$ 
where $\gamma = (1,2,\ldots, n)\in S_n$.

Another generalised BN-pair on $G=\GL$ is $(B,N)$ where $N$ is as above
and $B=Z(G)I$.
Indeed, $H = Z(G) \Dn (\O_\K^\times, \ldots, \O_\K^\times ) \cong \K^\times (\O_\K^\times)^n$
and
$N/H \cong \Sn \ltimes \Z_1^n$
where
$\Z_1^n = \Z^n/\langle (1,1,\ldots, 1) \rangle$.
The generalised Weyl group contains
the Weyl group $W=\Sn \ltimes \Z_0^n$ of type $\widetilde{A}_{n-1}$
as a normal subgroup of index $n$.
A complementary group can be chosen again as 
$\Omega = \langle (1,0,\ldots, 0) \cdot \gamma \rangle \cong \Cn$.
Finally,
$G_0 = BWB$ consists of those matrices whose determinant is
in $\langle\pi^n\rangle\O_\K^\times$ where $\pi\in \O_\K$
is a uniformizer.

Back to any group $G$ with a generalised BN-pair,
Lemma~\ref{BN_pro} guarantees not only
the existence of a building of $G_0$ of type $(W,S)$, say $\sB\sT$,
but also that $\sB\sT$ admits a well-defined simplicial $G$-action.
The fundamental apartment 
of $\sB\sT$ is the Coxeter complex associated to the Coxeter system $(W,S)$.
Hence, there exists a labelling which identifies each vertex of the fundamental chamber $C$ 
with an element of $S$.
We know that both $G$ and $G_0$ act on $\sB\sT$.
Let $G_1$ be the subgroup of $G$ that consists of all label-preserving elements. The following lemma  
summarises its properties:
\begin{lemma} The following statements hold in the notations above. 
\label{finite}
\begin{enumerate}
\item $G_1$ is a normal subgroup of $G$ containing $G_0$.
\item If $K$ is the kernel of the $G$-action on $\sB\sT$, then $G_1=KG_0$.
\item $(KB, N_1)$ is a BN-pair for $G_1$, where $N_1 = N \cap G_1$.
\item The buildings and the Weyl groups of $G_0$ and $G_1$ are the same.
\item $(KB,N)$ is a generalised BN-pair for $G$ with the same Weyl group $(W,S)$.
\item If $S$ is finite and the generalised Weyl group for the pair $(KB,N)$ is
$\Omega_1 \ltimes W$, then the constituent group $\Omega_1$ is finite.
\end{enumerate}
\end{lemma}
\begin{proof}
  (1) is obvious.

  To prove (2) pick $\vg_0\in G_0$ for any $\vg\in G_1$ as in Lemma~\ref{BN_pro}. These elements $\vg$ and $\vg_0$
  act in the same way on the set of chambers in $\sB\sT$. Since both preserve the labelling, they act on $\sB\sT$ in the same way
  and $\vg\in K\vg_0 \subseteq KG_0$.

  Once we know (2),   (3) and (4) follow from the label-preserving action of $G_1$ on $\sB\sT$,
  while (5) is a straightforward check of axioms. 

To prove (6), consider  $\vg, \vh \in G$ changing the labelling in the same way.
Then the element $\vg \vh^{-1}$  does not change the labelling and hence $\vg \vh^{-1} \in G_1$.
In other words, we have an injective map:
$$ \Omega_1 \cong G/G_1 \longrightarrow \Sn,$$
  where $n=\rank(\sB\sT)=|S|$.
\end{proof}


We will use the following adjectives for subgroups of $(W,S)$ and $G$:
\begin{itemize}
\item  A subgroup $W_J$ of $W$ is \emph{standard parabolic} if it is generated by 
some $J \subset S$. 
\item If a standard parabolic subgroup $W_J$ is finite, it is called \emph{spherical}. Ditto for the set $J$.
\item A subgroup $P_J$ of $G_0$ is called \emph{standard parabolic}
  if it is of the form $BW_JB$.
\item A subgroup of $G_0$ is called \emph{parabolic of type $J$} if it is conjugate to the standard parabolic subgroup $BW_JB$. It is called \emph{parabolic of finite type} if $W_J$ is spherical.
\end{itemize}

Denote by $\S$ the set of all spherical subsets of $S$ and consider the following set:
$$\sP \coloneqq \coprod_{J \in \S} G_0 / P_J.$$
This is a partially ordered set with respect to inclusion. Observe that
$\vg_0 P_{J_0} \leq \vg_1 P_{J_1}$ if $J_0 \sset J_1$ (hence $P_{J_0} \sset P_{J_1}$), and $\vg_0^{-1} \vg_1 \in P_{J_1}$.
Denote by $\sD_n$ the set of all chains of $\P$ of length $n+1$, $\sD_{(n)}\subseteq \sD_n$ the subset of proper chains:
$$
\sD_n = \{ \vg_0 P_{J_0} \subseteq \vg_1 P_{J_1} \subseteq \ldots \subseteq \vg_n P_{J_n} \}, \ \
\sD_{(n)} = \{ \vg_0 P_{J_0} \subset \vg_1 P_{J_1} \subset \ldots \subset \vg_n P_{J_n} \}.
$$
Then $\sD_\bullet = (\sD_n)$ is a simplicial set, whose geometric realisation $|\sD|$
is the geometric realisation of the poset $\sP$.
We call $\sD_\bullet$ \emph{the Davis building} of $G$.
The action of $G$ on the Bruhat-Tits building $\sB\sT$ induces
a simplicial action of $G$ on the Davis building $\sD_\bullet$.
\begin{lemma}
  \label{D_stab}
  Let $x = [\vg_0 P_{J_0} \subseteq \ldots \subseteq \vg_n P_{J_n}]\in \sD_n$. The stabiliser $G_x$
  is equal to
  $\vg_0 B \Omega_x W_{J_0} B \vg_0^{-1}$
where $\Omega_x = \bigcap_{i=0}^{n} \Om_{J_i}$ and $\Om_{J}$ is the stabiliser of $J$.
\end{lemma} 
\begin{proof}
  By the  definition of the partial order, for every $i \leq n$, there exists an element
  $\vp_i \in P_{J_i}$ with $\vg_{i-1}^{-1} \vg_i = \vp_i$.
Recursively we can write $\vg_i=\vg_0 \vp_1 \ldots \vp_i$. Hence
$$
(G_0)_{\vg_i P_{J_i}}
=\vg_i P_{J_i} \vg_i^{-1}=
\vg_0 \vp_1 \ldots \vp_i P_{J_i} \vp_i^{-1} \ldots \vp_1^{-1} \vg_0^{-1} = \vg_0 P_{J_i} \vg_0 ^{-1},
$$
since $P_{J_k} \sset P_{J_i}$ for all $k \leq i$. 
This allows us to
compute the stabiliser in $G_0$:
$$
(G_0)_x =
\bigcap_{i=0}^{n} (G_0)_{P_{J_i}}
=
\bigcap_{i=0}^{n} \vg_0 P_{J_i} \vg_0^{-1}
=\vg_0 P_{J_0} \vg_0^{-1}.
$$
Now, we move on to $G_x$. For every subgroup $P$ of $G$ containing $B$,
there exists a unique subset $J \sset S$ and a unique subgroup $\Om^\prime$ of $\Om$, 
such that $P=B \Om^\prime W_J B$ \cite{Iw}.
The subgroup $\vg_0^{-1}G_x \vg_0$ contains $B$, hence, it is one of these subgroups.
Moreover, as we know its intersection with $G_0$, we can conclude
that
$$
\vg_0^{-1}G_x \vg_0 = G_{\vg_0^{-1}\cdot x} =
B \Omega^\prime W_{J_0} B =
\bigcup_{\vu \in \Om'}  B \dot{\vu} W_{J_0} B
$$
for some subgroup $\Omega^\prime \leq \Omega$.
Clearly, $\vu \in \Omega^\prime$ if and only if
its lifting $\dot{\vu}$ stabilises all cosets in $\vg_0^{-1}\cdot x$, i.e., all $P_{J_i}$.
Thus, $\Omega^\prime = \bigcap_{i=0}^{n} \Om_{J_i}$. 
\end{proof}

We say that a topological group $G$ is {\em a topological group of Kac-Moody type}
if a generalised BN-pair
$(B,N,S)$ 
is selected such that
the following properties hold: 
\begin{itemize}
\label{p1}\item[(1)]  $G$ is a locally compact totally disconnected topological group. 
\label{p2}\item[(2)]  The set $S$ is finite.
\label{p3}\item[(3)]  The subgroup $B$ is open in $G$.
\label{p4} \item[(4)] The subgroup $B$ contains the kernel $K$ of the $G$-action on
the Bruhat-Tits building.
\label{p5} \item[(5)] If $J\subseteq S$ is a spherical subset, then
$P_J/K$ is compact.
\end{itemize}
Now we are ready for the main result of this section.

\begin{thm} \label{open}
A topological group $G$ of Kac-Moody type acts continuously on its Davis building $\sD_\bullet$.
Moreover, the stabiliser of each $x\in \sD_n$ is compact modulo the action kernel $K$.  
\end{thm}
\begin{proof}
The continuity of action is equivalent to all stabilisers $G_x$ being open.
This follows from Lemma~\ref{D_stab} and $B$ being open.
  
Since $B$ contains the kernel $K$, $G_1=G_0$ by Lemma~\ref{finite}.
Moreover, the subgroup $\Omega$ is finite. As $\sD_\bullet$ incorporates only spherical parabolic subgroups
of $G_0$, each stabiliser $G_x$ is union of finitely many double cosets $B\dot{\vw} B$.
Since $K$ is normal, $(B\dot{\vw} B)/K$ is the quotient topological space of $B/K\times B/K$.
Thus, each double coset $B\dot{\vw} B$ is compact modulo $K$ and so is $G_x$.
\end{proof}

The fundamental theorem of Davis is that if $S$ is finite, then $|\sD|$ is a CAT$(0)$ geodesic  space with
a piecewise Euclidean structure \cite{Davis2}. In particular, it is imperative for us 
that $|\sD|$ is contractible.
All conditions of Theorem~\ref{mainthm} and Theorem~\ref{local}
are satisfied.
Let us formulate them as a corollary.
Observe that the kernel $K$ contains any central subgroup,
so the condition $A\subseteq K$ holds automatically.
Observe also that $B/A$ is compact if and only if $K/A$ is compact.

\begin{cor}
\label{loc_KMgroup}
  Let $G$ be a topological group of Kac-Moody type, $A$ its central closed
  subgroup such that $B/A$ is compact. 
The localisation functor for the category of $A$-semisimple
  $G$-representations over a field $\F$
$$
\sM_A(G)
\xrightarrow{\cong} 
\ChD[\Sigma_A^{-1}] 
$$
is an equivalence of categories.
If the field  $\F$ is 
$G_{x}/A$-ordinary for any $x\in \sD_\bullet$, then
$$\quad \pd(\M_{A}(G)) \leq 
\underset{J \in \S}{\sup} |J|$$
where $|J|$ denotes the cardinality of $J$.
\end{cor}

We finish this section with another observation about the class
of groups we have introduced.
\begin{thm}\label{unimod}
A topological group of Kac-Moody type $G$ with compact $B$ is unimodular.
\end{thm}
\begin{proof}
We can use the compact open subgroup $B$ in Proposition~\ref{unimod_cr} to compute the modular function.
In particular, $\Delta (\vx)=1$ for all $\vx\in B$.
Part (v) of the definition of a generalised BN-pair 
ensures that $\Delta (\dot{\va})=1$ for all $\va\in\Omega$. 
If $\vs\in S$, then $\dot{\vs}^{-1}B\dot{\vs} = \dot{\vs}B\dot{\vs}^{-1}$, so again $\Delta (\dot{\vs})=1$.
  
The theorem follows because $B$, $\dot{S}$ and $\dot{\Omega}$ generate $G$.
\end{proof}

\section{Topological Kac-Moody Groups}

There are several versions of 
complete Kac-Moody groups in the
literature. 
The groups described in Kumar's book \cite{Kum}
are ind-algebraic.
They are not locally compact, so of little relevance to  our
investigation.
There are several locally compact Kac-Moody groups including 
{\em Caprace--R\'{e}my--Ronan groups}, 
{\em Carbone--Garland--Rousseau groups}
and 
{\em Kumar--Mathieu--Rousseau groups}. 
A good review of various relevant complete Kac-Moody groups
can be found in Marquis' thesis \cite{Mar}.
A  paper by Capdeboscq and Rumynin \cite{CaRu} 
contains a general approach to these groups including
a construction of a new class of
{\em locally pro-$p$-completed groups}.

Let $\sA=(\alpha_{i,j})_{n\times n}$ be a generalised Cartan matrix,
$(W,S)$ its Weyl group, 
$\mD=(X,Y,\Pi,\Pi^\vee )$  a root datum of type $\sA$.
Following Carter and Chen \cite{CaCh} 
we can define a Kac-Moody group
${G}_\mD (\K)$
over 
a field $\K$.
The topological Kac-Moody is a certain
completion $\widehat{{G}_\mD (\K)}$.
We refer the reader to \cite{CaRu} (also cf. \cite{Mar})
for further details. 
If the field $\K$ is finite,
the group $\widehat{{G}_\mD (\K)}$ can be locally compact.
It acts on a building of type $(W,S)$.
The kernel of this action  $K$ is central for some completions;
for some other completions very little is known about $K$.
By choosing an appropriate subgroup $K_0\leq K$,
we can derive examples of topological groups of Kac-Moody type
in the form $G\coloneqq\widehat{{G}_\mD (\K)}/K_0$.
The following proposition summarises what we know about their
representations from 
Corollary~\ref{loc_KMgroup}: 

\begin{prop}
  Let $G=\widehat{{G}_\mD (\K)}/K_0$ be a topological group of Kac-Moody type
  derived as described in this section, $A$ its central closed
  subgroup such that $B/A$ is compact.
  The localisation functor for the category of $A$-semisimple
  $G$-representations over a field $\F$
$$
\sM_A(G)
\xrightarrow{\cong}
\ChD[\Sigma_A^{-1}] 
$$
is an equivalence of categories.
If the field  $\F$ is 
$P_{J}/A$-ordinary for any spherical $J\subseteq S$, then 
$$\quad \pd(\M_{A}(G)) \leq 
\underset{J \in \S}{\sup} |J| = f(\sA),$$
where $f(\sA)$ is the maximal size of the diagonal minor of $\sA$ of finite type.
\end{prop}

Let us call a generalised Cartan matrix $\sA$
{\em generic} if $f (\sA)=1$.
Thus, for a generic $\sA$, we obtain hereditary abelian categories.
It would be interesting to investigate them further.

Another direction for further research is Schneider-Stuhler
resolutions in $\M_{A}(G)$ for topological Kac-Moody groups.
We are going to address them in consequent papers.

\section{Homological Duality}

We start with a locally compact totally disconnected group $G$
and its closed central subgroup $A$.
We make no restriction on $\F$ for now.

We consider one of the derived categories
$D^\star(\sM(G))$ where $\star\in\{\mbox{``empty''},-,+,b\}$.
We have been working with chain complexes previously,
but we feel obliged to switch to cochain complexes
at this point to follow standard conventions.
Let us consider  a full subcategory $D^\star (\sM(G))_{A,\chi}$
for each character $\chi$ of $A$.
It consists of cochain complexes
$M^\bullet = (M^n,d^n)$ such that for all $\va\in A$
we have an equality
$\va-\chi(\va)=0$ in $\Hom (M^\bullet, M^\bullet)$.
This enables us to define a full subcategory
$D^\star (\sM(G))_A \coloneqq \oplus_\chi D^\star (\sM(G))_{A,\chi}$
consisting of $A$-semisimple complexes. 
There are two further related categories:
a full subcategory $D^\star_A(\sM(G))$ of complexes
with $A$-semisimple cohomology
and
$D^\star (\sM_A(G))$. 
The natural functors
$D^\star(\sM_A(G))\rightarrow D^\star(\sM(G))_A$
and 
$D^\star (\sM(G))_A\rightarrow D^\star_A(\sM(G))$
are not equivalences, in general.
It is a moot point when they are
(cf. \cite[Exercises in III.2]{GeMa}).

Let $B^\bullet$ be a complex of 
$G$-$G$-bimodules, smooth as both
left and right $G$-modules
such that the left and the right actions of $A$ on $B^\bullet$ coincide. 
We denote these actions on $B^n$ by $\,^\vg b$ and $b^\vg$.
The bimodule $B$ defines ``a dual module'' for each $M^\bullet \in D(\sM(G))$
by
$$
\nabla (M^\bullet) = \nabla_{B^\bullet}(M^\bullet) \coloneqq \Hom (M^\bullet, B^\bullet), \ \ 
[\vg\cdot \varphi] (m) \coloneqq (\varphi (m))^{\vg^{-1}}.
$$
Observe that the functor $\nabla$ preserves $D(\sM(G))_A$
because 
the image of $\varphi$ necessarily takes values in the $A$-socle of $B^\bullet$.
In fact, $\nabla$ takes $D(\sM(G))_{A,\chi}$ to $D(\sM(G))_{A,\chi^{-1}}$. 
The preservation of other categories depends on $B^\bullet$.
We say that $B^\bullet$ is {\em dualising} if
$\nabla$ restricts to a self-equivalence
$D^b(\sM^{f.g.}(G))_A\rightarrow D^b(\sM^{f.g.}(G))_A$
of the derived categories of finitely generated modules
and $\nabla^2$ is naturally isomorphic to $\mbox{Id}_{D^b(\sM^{f.g.}(G))_A}$.

It would be extremely interesting to develop a theory of dualising complexes
in our generality in the spirit of Hartshorne \cite{Har}
and, in particular, characterise the dualising complexes
as done for rings by Yekutieli \cite[Def. 4.1]{Yek}.
\begin{prop} \label{dual_H}
  (cf. \cite[Th. 31]{Bern})
  Suppose that the field $\F$ is $K$-ordinary
  for a compact open subgroup $K$ of $G$.
  Then the Hecke algebra
  $\sH=\sH(G,\F,\mu_K)$
  is a dualising bimodule.
\end{prop}
\begin{proof}
Thanks to Proposition~\ref{equiv1} we are dealing
with modules over the idempotented algebra $\sH$.
An object $M^\bullet \in D^b(\sM^{f.g.}(G))_A$
admits a projective resolution
$P^\bullet = (P^n,d^n) \cong M^\bullet$
in $K^-(\sM(G))$, i.e., $P_n =0$ for $n \gg 0$.  
Each $P_n$ can be chosen to be a finite direct sum of $\sH e$ for various idempotents $e$.

We can compute $\nabla (M^\bullet)$ on this resolution.
The natural action of $G$ on $\nabla (M^\bullet)$ is the
right actions $[\varphi\leftharpoonup \vg] (m) \coloneqq (\varphi (m))^{\vg}$
that we turn into the left action using the inverses.
Let us not do it so that we can treat $\nabla (P^\bullet)$
as a complex of right $\sH$-modules. In particular, we can use the natural
isomorphism
$$
\nabla (\sH e) = \Hom_\sH (\sH e, \sH) \cong e \sH, \ \ \
F \longleftrightarrow F(e)=eF(e)
$$
to construct the natural isomorphism of functors
$$
\nabla^2 \stackrel{\gamma}{\Longrightarrow} \mbox{Id}_{D^b(\sM^{f.g}(G))_A}, \ \ \
\nabla^2 (P^n) = \nabla^2 (\oplus_e \sH e) \xrightarrow{\cong}
\nabla (\oplus_e e\sH ) \xrightarrow{\cong}
\oplus_e \sH e = P^n.
$$
To show that it is well-defined
we need to compute what happens to differentials $d^n$.
Each differential is a matrix $(d_{e,f}^n)$ where
$d^n_{e,f}\in\Hom_\sH (\sH e, \sH f)$.
Using natural isomorphisms
$$
\Hom_\sH (\sH e, \sH f) \cong e \sH f, \
F \longleftrightarrow F(e)
\ \mbox{ and } \ \
\Hom_\sH (f \sH , e \sH) \cong e \sH f, \
F \longleftrightarrow F(f),  
$$
we can write each $d_{e,f}^n$
as $e \Theta_{e,f} f$ for some
$\Theta_{e,f} \in\sH$ that helps us to perform the key calculation:
$$
\nabla^2 \big( (d_{e,f}^n) \big)=
\nabla^2 \big( (e \Theta_{e,f} f) \big)=
\nabla \big( (e \Theta_{e,f} f) \big )=
(e \Theta_{e,f} f) =
(d_{e,f}^n).
$$
Naturality of the transformation $\gamma$ is apparent
after this calculation.
Finally, $\nabla$ is an equivalence because its quasi-inverse is itself.
\end{proof}

We would like to state the following conjecture.
It is known for $p$-adic reductive groups \cite[III.3]{SS2}.
It is obvious if the Davis building $\sD$ is a tree because
$\nabla_\sH (M)$ is necessarily quasiisomorphic
to the sum of its cohomologies as a consequence of projective
dimension one.
\begin{conj} \label{conj_H}
  Suppose we are under the assumptions of Proposition~\ref{dual_H}
  and $M\in \sM_A(G)$ is a simple module.
  Then $\nabla_\sH (M)$ is a complex with cohomologies in one degree.
\end{conj}

If Conjecture~\ref{conj_H} holds, we can write
$\nabla_\sH (M) \cong M^\vee [d(M)]$
in the derived category. 
Both the module $M^\vee$ and the integer $d(M)$
are of exceptional interest. It is easy to show that
$M^\vee$ is also a simple module.
We finish the paper with a conjectural description
of the homologically dual module $M^\vee$ for topological groups of Kac-Moody type that
agrees with the known description for $p$-adic groups \cite{EvMi}.

Let $G$ be a topological group of Kac-Moody type as
defined in Section~\ref{real}.
We make additional assumptions for simplicity:
\begin{enumerate}
\item $B$ is compact,
\item $\F$ is a $B$-ordinary field (so we can choose $\mu=\mu_B$),
\item $(B,N,S)$ is a BN-pair on $G$,
\item $A$ is trivial.
\end{enumerate}
Assumption (1) can be achieved for an arbitrary topological group
of Kac-Moody group $G^\prime$ by replacing it with $G=G^\prime / K^\prime$
where $K^\prime$ is the kernel of $(\rho,M)\in\sM (A)$.
Assumption (3) can be achieved by restricting $(\rho,M)$ to $G_0$.

Let us denote $\sH( B\setminus G/B)$ 
the space of $\F$-valued compactly supported
$B$-bi-invariant functions on $G$.
This space is a  subalgebra of the Hecke algebra $\sH(G,\F,\mu)$.
For each element $\vw$ of the Weyl group $W$
we denote by $\Theta_\vw$  the delta-function of
the double coset $B \dot{\vw} B$, i.e.,
$\Theta_\vw (\vx)=1$ if $\vx\in B\dot{\vw}B$
and $\Theta_\vw (\vx)=0$ otherwise. Clearly, 
$\Theta_\vw, \, \vw \in W$ form an $\F$-basis of
{\em the spherical Hecke algebra} $\sH( B\setminus G /B)$.

We should relate the spherical Hecke algebra
to {\em the multiparameter Iwahori-Hecke algebra} $\H[q_\vs, q_\vs^{-1}]$ \cite{Geck}.
The formal variable $q_\vs,\,\vs \in S$ depends only on the $W$-conjugacy
class of $\vs$: we set $q_\vs=q_\vt$ if there exists $\vw\in W$
such that $\vs=\vw\vt\vw^{-1}$. 
Then $\H[q_\vs, q_\vs^{-1}]$ is a $\Z[q_\vs, q_\vs^{-1}]$-algebra
generated by elements $T_\vs$, for $\vs \in S$, which satisfy the following relations:
\begin{enumerate}
\item $T_{\vs} T_{\vt} T_{\vs} \ldots = T_{\vt} T_{\vs} T_{\vt} \ldots \ \ \ $
  for all $\vs \neq \vt \in S$
  with the element $\vs \vt$ of finite order
  where each side of the equality contains exactly $|\vs\vt|$ $T$-s.
\item $(T_\vs - q_\vs)(T_\vs +1)=0\ \ \ $ for all $\vs \in S$.
\end{enumerate}
The relation between these two algebras is summarised
in the following proposition, whose
proof is standard.
\begin{prop}
  The natural homomorphism
  $$\H[q_\vs, q_\vs^{-1}]\mathbin{ \underset{ \mathclap{ \Z[q_\vs, q_\vs^{-1}] } }{\otimes } } \F
  \longrightarrow \sH (B\setminus  G /B), \ \ \ 
  T_\vs \otimes 1 \mapsto \Theta_\vs, \ \ \
  q_\vs \mapsto |B:B\cap \dot{\vs} B \dot{\vs}^{-1}| \cdot 1_\F
  $$
  is an isomorphism of algebras.
\end{prop}
Let us use this isomorphism to define a new involution.
Consider the antipode map on the Hecke algebra from Section~\ref{zero}: 
$$
\sigma : \sH \rightarrow \sH, \ \ 
\sigma (\Theta) (\vx) =    \Theta (\vx^{-1})
\ \mbox{ for all } \ \vx \in G.
$$
On the level of Iwahori-Hecke algebra the antipode is
a $\Z[q_\vs, q_\vs^{-1}]$-linear antihomomorphism of
$\H[q_\vs, q_\vs^{-1}]$ such that $\sigma (T_\vs) = T_\vs$.
We define {\em Iwahori-Matsumoto involution (antiinvolution)} 
as a $\Z[q_\vs, q_\vs^{-1}]$-linear homomorphism (antihomomorphism) of
$\H[q_\vs, q_\vs^{-1}]$ such that
$$
\iota_{IM} (T_\vs) = -q_\vs {T_\vs}^{-1} \ \ \ 
(\sigma_{IM} (T_\vs) = -q_\vs {T_\vs}^{-1} \ \ 
\mbox{ correspondingly}). 
$$
Observe that the four maps $\mbox{Id},\sigma, \sigma_{IM}, \iota_{IM}$ form a Klein four-group.
Now we use the functor $\sF$ and the idempotent $\Lambda_B$
(cf. Prop.~\ref{equiv1})
to formulate the final conjecture of our paper.
It is known for the $p$-adic reductive groups \cite{EvMi}.

\begin{conj} \label{conj_IM}
  Suppose we are under the additional assumptions (1)-(4) stated above
  and $M$ is a simple module in $\sM(G)^B$.
  Then $M^\vee$ is also a simple module in $\sM(G)^B$
  and the $\sH (B\setminus  G /B)$-modules
  $\Lambda_B \ast \sF (M)$ and
  $\Lambda_B \ast \sF (M^\vee)$ are twists of each other
  with respect to the Iwahori-Matsumoto involution $\iota_{IM}$.
\end{conj}

\end{document}